\documentclass[journal,10pt]{IEEEtran}

\usepackage{ifpdf}

\newcommand\MYhyperrefoptions{bookmarks=true,bookmarksnumbered=true,
	pdfpagemode={UseOutlines},plainpages=false,pdfpagelabels=true,
	colorlinks=true,linkcolor={black},citecolor={black},urlcolor={black},
	pdftitle={},
	pdfsubject={},
	pdfauthor={Mohammadhafez Bazrafshan and Nikolaos Gatsis},
	pdfkeywords={}}

\ifCLASSINFOpdf
\usepackage[\MYhyperrefoptions,pdftex]{hyperref}
\else
\usepackage[\MYhyperrefoptions,breaklinks=true,dvips]{hyperref}
\usepackage{breakurl}
\fi

\usepackage[noadjust]{cite}

\usepackage{color}
\usepackage{xcolor}

\ifCLASSINFOpdf
  \usepackage[pdftex]{graphicx}
  \usepackage{epstopdf}
     \DeclareGraphicsExtensions{.eps,.pdf}
\else
 \usepackage[dvips]{graphicx}
   \DeclareGraphicsExtensions{.eps}
\fi

\usepackage{amsthm}
\theoremstyle{plain}

\newtheorem*{theorem*}{\textbf{Theorem}}
\newtheorem{theorem}{\textbf{Theorem}}
\newtheorem*{definition*}{\textbf{Definition}}
\newtheorem{lemma}{Lemma}
\theoremstyle{remark}
\newtheorem*{remark*}{Remark}

\newcommand{\bmat}[1]{\begin{bmatrix} #1 \end{bmatrix}}

\usepackage{amsmath}
\usepackage{amssymb}    
\usepackage{multirow}

\usepackage[nodisplayskipstretch]{setspace} 
\newcommand{\mr}[1]{\mathrm{#1}}
\newcommand{\diag}{\mathrm{diag}}
\newcommand{\mb}[1]{\mathbf{#1}} 
\newcommand{\bs}[1]{\boldsymbol{#1}}  
\newcommand{\mbb}[1]{\mathbb{#1}}
\newcommand{\mc}[1]{\mathcal{#1}}

\newlength{\eqboxstorage}

\ifCLASSOPTIONcompsoc
\usepackage[caption=false,font=normalsize,labelfon
t=sf,textfont=sf]{subfig}
\else
\usepackage[caption=false,font=footnotesize]{subfi
g}
\fi

\usepackage{url}
\begin{document}
\makeatother
\title{Convergence of the Z-Bus Method for Three-Phase Distribution Load-Flow with ZIP Loads}

\author{Mohammadhafez~Bazrafshan,~\IEEEmembership{Student Member,~IEEE,}~and~%
        Nikolaos~Gatsis,~\IEEEmembership{Member,~IEEE}
        \thanks{Manuscript received May 26, 2016; revised November 22, 2016 and March 1, 2017; accepted April 26, 2017.  The authors are with the Department of Electrical and Computer Engineering, The University of Texas at San Antonio. E-mails: \{mohammadhafez.bazrafshan, nikolaos.gatsis\}@utsa.edu.  This material is based upon work supported by the National Science Foundation under Grant No. CCF-1421583.}}

\markboth{IEEE Transactions on Power Systems (Accepted)}%
{Bazrafshan and Gatsis}


\maketitle

\begin{abstract}
This paper derives a set of sufficient conditions guaranteeing that the load-flow problem in unbalanced three-phase distribution networks with wye and delta ZIP loads has a unique solution over a region that can be explicitly calculated from the network parameters.  It is also proved that the well-known Z-Bus iterative method is a contraction over the defined region, and hence converges to the unique solution. 
\end{abstract}


\IEEEpeerreviewmaketitle

\section{Introduction}

\IEEEPARstart{E}{xistence} of three-, two-, and one-phase transmission lines with high $R/X$ ratios in distribution networks is predominant. Therefore, the Z-Bus iterative method  is preferred over classical Newton-Raphson methods  for distribution load-flow~\cite{Chen1991pf}. 
This paper deals with  the Z-Bus method when applied to practical distribution networks including three-phase wye and delta loads with constant-power, constant-current, and constant-impedance portions (ZIP loads), as well as potentially untransposed lines.  Conditions on the network components and loads guaranteeing the existence of a unique load-flow solution and the convergence of the Z-Bus method are derived. 

The availability of conditions guaranteeing power flow solution existence has several applications in power system planning and operations. In planning studies, such as conventional and distributed generation planning, conditions for solution existence serve as assurance that desirable operating conditions are feasible under future load demands. For example, such conditions provide network-specific guides on whether a certain lateral will have valid load-flow solutions over the course of years, as the loads increase \cite{LesieutreSauerPai1999, Molzahn2013}.

In operations, conditions for solution existence allow power engineers to determine whether the set of power injections and flows respect security constraints  of power system equipment  ~\cite{WuKumagai1982}.  Furthermore, the availability of solvability regions in terms of nodal voltages renders improved initial estimates for numerical algorithms and accelerates the online analysis required for real-time monitoring and control applications~\cite{bolognani2016}.

Specifically in regards to distribution networks, their weakly-meshed structure limits the number of feasible voltage profiles. This characteristic implies that it is harder to meet different demand patterns in  distribution networks \cite{ChiangBaran1990}. In this case, explicit sufficient conditions for solution existence that take into account various load types---that is, wye and delta ZIP loads---prove to be very useful.

The literature on existence of solutions to the power flow equations is extensive; see e.g., \cite{Molzahn2013} for a review on general power networks, and  recent works in~\cite{bolognani2016, YuTuritsyn2015,WangBernsteinBoudecPaolone2016} geared toward distribution networks.  Here, the prior art most closely related to the present work is reviewed. 

The single-phase equivalent of a power distribution network  is considered  in \cite{bolognani2016}.  Power flow equations are formulated so that  injection currents can be expressed as a fixed point of a quadratic map parametrized by constant-power loads and the network admittance matrix.  Sufficient conditions on the loads and the network admittance matrix are derived to guarantee existence of a power flow solution.

The work in \cite{YuTuritsyn2015} highlights that the injection currents can be expressed by a family of fixed-point quadratic maps  similarly to \cite{bolognani2016} but parametrized by a diagonal matrix. An algorithm is proposed for finding the diagonal matrix that expands the regions given in \cite{bolognani2016} where a power flow solution exists.

Convergence of the Z-Bus method interpreted as a fixed-point iteration is considered in  \cite{WangBernsteinBoudecPaolone2016} and \cite{ZhaoChiangKoyanagi2016} for single-phase distribution networks. Specifically, explicit conditions on constant-power loads and the network admittance matrix are presented in \cite{WangBernsteinBoudecPaolone2016} ensuring convergence of the Z-Bus method and existence of a unique solution for distribution load-flow under balanced assumptions. Interestingly, it is numerically verified in \cite{ZhaoChiangKoyanagi2016} that when distributed generation units are modeled as constant-power nodes, the voltage iterates produced by the Z-Bus method are indeed contracting.   

Single-phase power system modeling is valid under balanced three-phase operation, which is a  reasonably accurate assumption  for transmission networks.  In distribution networks, there is a prevalence of untransposed three-, two- and one-phase transmission lines  with unbalanced loading.  There is a gap in the literature, since theoretical results on solution existence and convergence of load-flow methods  overlook the underlying unbalanced nature of distribution systems.

This paper considers general unbalanced distribution networks that include transmission lines with one, two, or three available phases as well as wye and delta ZIP loads,  in contrast to \cite{bolognani2016,YuTuritsyn2015,WangBernsteinBoudecPaolone2016}, which only consider single-phase networks with constant-power loads. Sufficient conditions for the Z-Bus method  to be a contraction over a region explicitly defined in terms of the loads and the bus admittance matrix are derived. A consequence of the contraction is that power-flow equations are guaranteed to have a unique solution in the defined region.
 
In order to prove that the Z-Bus method is  a contraction mapping,  a candidate region of contraction is defined and the self-mapping and contraction properties are proved. These steps were followed in \cite{WangBernsteinBoudecPaolone2016} for single-phase networks with constant-power loads. However, the extension to  general unbalanced three-phase networks with wye and delta ZIP loads and missing phases adds a level of complexity that necessitates the development of novel methods to prove the convergence of the Z-Bus method.

The major contribution of this paper is to derive sufficient convergence conditions that handle the complexities arising from various combinations of load types and phasings in practical three-phase distribution networks. The conditions are expressed directly in terms of the network loads and admittance matrix, and guarantee the existence of a unique solution to the power flow equations. 
It becomes immediately apparent from these conditions that different load types contribute differently to the network solvability.  Moreover, the generality of the network model renders results that are readily applicable to  practical distribution networks without simplifying assumptions. To this end, the conditions are validated on several IEEE distribution test feeders.

In order to prove the aforementioned conditions, this paper develops a set of tools that advance the power flow analysis via contraction methods. These tools are particularly tailored to handle wye and delta connections, ZIP loads, and missing phases, and are expected to be useful to other researchers working in this area. In addition, this paper constructs interesting and non-trivial examples of networks where the Z-Bus method exhibits oscillatory behavior, and the conditions are not satisfied. One example amounts to a two-node network where the non-slack bus has constant-current and –impedance loads, but a constant-power injection (e.g., from a renewable energy source). The second example features a nontrivial unbalanced network with phase couplings in its lines.

\emph{Paper Organization:} The network and the ZIP load models with wye and delta configurations are introduced in  Section \ref{sec:model}. The Z-Bus method is briefly reviewed in Section~\ref{sec:zbus}. Section~\ref{sec:conditions}  presents the main theorem, which contains the sufficient conditions for the Z-bus method to be a contraction.  The theorem is verified numerically in Section \ref{sec:numericaltests}.  Section \ref{sec:conclusion} provides pointers to future work, and Appendices~\ref{sec:app:voltagebounds} and~\ref{sec:appendixProof} include a detailed proof of the main result.

\emph{Notation:} For a vector $\mb{u}$, $u_j$ denotes the $j$-th element, and $\diag(\mb{u})$ is a square matrix with elements of $\mb{u}$ on the main diagonal. The notation $\mb{v}=\{\mb{v}_n\}_{n \in \mc{N}}$ is used to denote a vector constructed from vertically stacking vectors $\mb{v}_n$ for $n \in \mc{N}$.  For a matrix $\mb{Z}$,   $\mb{Z}_{\bullet k}$ denotes the $k$-th column,   $\mb{Z}_{j\bullet}$ denotes the $j$-th row, $\mb{Z}_{jk}$ denotes the element in the $j$-th row and $k$-th column; and $[\mb{Z}_{\bullet k} ]_{k \in \mc{N}_Y}$ denotes a matrix constructed from the columns of $\mb{Z}$ in the set $\mc{N}_Y$.  The infinity norm for a vector $\mb{u}$ and a matrix $\mb{Z}$ is defined as $\|\mb{u}\|_{\infty}=\max_{j} |u_j|$ and  $\|\mb{Z}\|_{\infty}= \max_{j} \sum_{k} |\mb{Z}_{jk}|$, respectively.

\section{Network and load model}
\label{sec:model}

Motivated by  the various three-phase connections in  IEEE distribution test feeders, a detailed  network model is presented in this section. 

\subsection{Three-phase network modeling preliminaries}
A  power distribution network is represented by a graph $(\mc{N}, \mc{E})$, where $\mc{N}:= \{1, 2, \ldots, N\} \cup \{\mr{S}\}$ is the set of nodes  and  $\mc{E} \subset \mc{N} \times \mc{N} :=\{ (n,m) | n,m \in \mc{N}\}$ is the set of edges. Node $\mr{S}$ is considered the slack bus, i.e., the point of interconnection to the transmission network.   Furthermore, let $\mc{N}$ be partitioned as $\mc{N}=\mc{N}_Y \cup \mc{N}_\Delta \cup \{\mr{S}\}$ where $\mc{N}_Y$ and $\mc{N}_\Delta$ collect nodes with wye and delta connections, respectively.  

The set of available phases at node $n$ is denoted by $\Omega_n \subseteq \{a,b,c\}$.  For wye nodes, i.e.,  $n \in \mc{N}_Y$, $\Omega_n$ may have one, two, or three phases.  For delta nodes, i.e., $n \in \mc{N}_\Delta$,  $\Omega_n$ has at least two  phases. For future use, define $r(\phi)$ as the right shift of phase $\phi$ as follows: 
\begin{IEEEeqnarray}{rCl}
r(a)=b, r(b)=c, r(c)=a. 
\end{IEEEeqnarray}
 If $|\Omega_n|=3$, then we have that $r(\phi) \in \Omega_n$ for all phases $\phi \in \Omega_n$.  If $|\Omega_n|=2$, then there exists only one phase $\phi \in \Omega_n$ such that $r(\phi) \in \Omega_n$.  For example, if $\Omega_n=\{a,c\}$,  only phase $c$ has the property that $r(c)=a \in \Omega_n$.

For every node $n \in \cal N$, the complex line to neutral voltage vector  is denoted by $\mb{v}_n:= \{v_{n}^\phi\}_{\phi \in \Omega_n} \in \mbb{C}^{|\Omega_n|}$.  The slack bus voltage is fixed at $\mb{v}_{\mr{S}}=\{1, e^{-j\frac{2\pi}{3}}, e^{j\frac{2\pi}{3}}\}$.  The  line to neutral voltages are collected in a vector $\mb{v}=\{\mb{v}_n\}_{n \in \mc{N} \setminus \{\mr{S}\}}$. Moreover, we use a vector  $\mb{e}_n^{\phi \phi'} \in \mbb{C}^{|\Omega_n|}$ with entries of 1, -1, and possibly 0, such that $(\mb{e}_n^{\phi \phi'})^T\mb{v}_n=v_n^{\phi}- v_n^{\phi'}$ gives the line to line voltages between phases $\phi$ and $\phi'$.   It is easy to see that $(\mb{e}_n^{\phi \phi'})^T\mb{v}_n=-(\mb{e}_n^{\phi' \phi})^T\mb{v}_n$.    

Define further an index set $\mc{J}:=\{1, \ldots, J\}$ where $J=\sum_{n=1}^N |\Omega_n|$ is the length of vector $\mb{v}$, and each $j\in \mc{J}$ is a linear index corresponding to a particular pair $(n, \phi)$ with $n \in \mc{N} \setminus \{\mr{S}\}$ and $\phi \in \Omega_n$. In this case, we write the operation $j=\texttt{Lin}(n,\phi)$. The operation $\texttt{Node}[\texttt{Lin}(n,\phi)] = \texttt{Node}[j] = n$ relates the index $j \in \mc{J}$ to the corresponding node $n$.  Define the set $\mc{J}_n:=\{j \ | \ \texttt{Node}[j]=n\}$ as the set of all linear indices $j$ that correspond to node $n$.  Finally,  set $\mc{J}$ is partitioned as $\mc{J}= \mc{J}_Y \cup \mc{J}_{\Delta}$, where $\mc{J}_Y= \bigcup_{n \in \mc{N}_Y} \mc{J}_{n}$ and $\mc{J}_{\Delta} = \bigcup_{n \in \mc{N}_\Delta} \mc{J}_n$.

\subsection{Three-phase load models}
For a  load at node $n$, the complex  vector of net current injections is denoted by $\mb{i}_n=\{i_n^\phi\}_{\phi \in \Omega_n}$ and we define $\mb{i}= \{ \mb{i}_n\}_{n \in \mc{N}\setminus \{\mr{S}\}}$.  Due to existence of loads,   the nodal net current injection $\mb{i}$ is a function of  nodal voltages $\mb{v}$. This dependence is denoted by $\mb{i}_n(\mb{v}_n)$. According to the ZIP load model,   $\mb{i}_n(\mb{v}_n)$  is  composed of  currents from  constant-power loads $\mb{i}_{\mr{PQ}_n} =\bigl\{ i_{\mr{PQ}_n}^{\phi} \bigr\}_{\phi \in \Omega_n}$, constant-current loads $\mb{i}_{\mr{I}_n}= \bigl\{i_{\mr{I}_n}^{\phi}\bigr\}_{\phi \in \Omega_n}$, and constant-impedance loads $\mb{i}_{\mr{Z}_n}=\bigl\{ i_{\mr{Z}_n}^{\phi}\bigr\}_{\phi \in \Omega_n}$.  For  $n \in \mc{N} \backslash \{\mr{S}\}$ and $\phi \in \Omega_n$ we have that
\begin{IEEEeqnarray}{rCl}
i_{n}^\phi(\mb{v}_n) =  i_{\mr{PQ}_n}^{\phi}(\mb{v}_n) + i_{\mr{I}_n}^{\phi} (\mb{v}_n)+ i_{\mr{Z}_n}^{\phi}(\mb{v}_n) \label{eqn:netI}
\end{IEEEeqnarray}
where functions $i_{\mr{PQ}_n}^{\phi}(\mb{v}_n)$, $i_{\mr{I}_n}^{\phi}(\mb{v}_n)$, and $i_{\mr{Z}_n}^\phi(\mb{v}_n)$  are defined in \eqref{eqnGroup:wyeloads} for wye connections  and in  \eqref{eqnGroup:deltaloads} for delta connections.  

Specifically, for  wye connections  ($n \in \mc{N}_Y$ and $\phi \in \Omega_n$), the ZIP load components are
\begin{subequations}
\label{eqnGroup:wyeloads}
\begin{IEEEeqnarray}{rCl}
 i_{\mr{PQ}_n}^{\phi}(\mb{v}_n) = -(s_{L_n}^{\phi} / {v_n^\phi})^* \label{eqn:wyeloadPQ} \\
 i_{\mr{I}_n}^{\phi}(\mb{v}_n) =  - \frac{v_n^{\phi}}{|v_n^{\phi}|} i_{L_n}^{\phi},  \label{eqn:wyeloadI} \\
 i_{\mr{Z}_n}^{\phi}(\mb{v}_n) = - y_{L_n}^{\phi}v_n^\phi,  \label{eqn:wyeloadZ}
\end{IEEEeqnarray}
\end{subequations}
where $s_{L_n}^\phi$ is the nominal power in the constant-power portion of the ZIP model, $i_{L_n}^\phi$ is the nominal current for the constant-current portion of the ZIP model, and $y_{L_n}^\phi$ is the nominal admittance in the  constant-impedance portion of the ZIP model at node $n \in \mc{N}_{Y}$ and phase $\phi \in \Omega_n$.  

For delta connections ($n \in \mc{N}_\Delta$ and $\phi \in \Omega_n$), the ZIP load components are
\begin{subequations}
\label{eqnGroup:deltaloads}
\begin{IEEEeqnarray}{rCll}
i_{\mr{PQ}_n}^{\phi}(\mb{v}_n) &=&- \sum\limits_{ \phi' \in \Omega_n \backslash \{ \phi \}} 
\left(\frac{ s_{L_n}^{\phi \phi'}}{(\mb{e}_{n}^{\phi \phi'})^T\mb{v}_n}\right)^*  \label{eqn:deltaloadPQ} \\
i_{\mr{I}_n}^{\phi}(\mb{v}_n)&=&-  \sum\limits_{ \phi' \in \Omega_n \backslash \{ \phi \}} i_{L_n}^{\phi \phi'} \frac{(\mb{e}_{n}^{\phi \phi'})^T\mb{v}_n}{|(\mb{e}_{n}^{\phi \phi'})^T\mb{v}_n|} \label{eqn:deltaloadI}\\
i_{\mr{Z}_n}^{\phi}(\mb{v}_n)&=& -\sum\limits_{ \phi' \in \Omega_n \backslash \{ \phi \}}   y_{L_n}^{\phi \phi'} (\mb{e}_{n}^{\phi \phi'})^T\mb{v}_n, \label{eqn:deltaloadZ}\IEEEeqnarraynumspace  
\end{IEEEeqnarray}
\end{subequations}
where $s_{L_n}^{\phi \phi'}$, $i_{L_n}^{\phi \phi'}$ , and $y_{L_n}^{\phi \phi'}$  are respectively the nominal power,  current, and admittance in the ZIP model for nodes $n \in \mc{N}_\Delta$ and over phases $\phi, \phi' \in \Omega_n$. For $n \in \mc{N}_{\Delta}$ and $\phi, \phi' \in \Omega_n$, we have that $s_{L_n}^{\phi \phi'} = s_{L_n}^{\phi' \phi}$, $i_{L_n}^{\phi \phi'}= i_{L_n}^{\phi'\phi}$, and $y_{L_n}^{\phi \phi'}=y_{L_n}^{\phi'\phi}$.

Notice that the constant-current loads in \eqref{eqn:wyeloadI} and \eqref{eqn:deltaloadI} rightfully adhere to the proper definition given in \cite{loadmodels}, where it is emphasized that the power varies directly with voltage magnitude.  To see this, one can calculate  the apparent power consumption for the load model in \eqref{eqn:wyeloadI} which gives  $s_n^{\phi}=v_n^{\phi} [({v_n^{\phi}}/{|v_n^{\phi}|}) i_{L_n}^{\phi}]^* = |v_n^\phi| (i_{L_n}^{\phi})^*$, i.e., the apparent power is only a function of the magnitude of the voltage $|v_n^{\phi}|$. It is easy to see that in this case, the load current magnitude is $|i_{L_n}|$ and the load power factor is $\frac{\mr{Re}[i_{L_n}^\phi]}{|i_{L_n}^{\phi}|} $, both of which are constant values, concluding that these models are  in line with the definition in \cite[pp. 315]{KerstingBook2001} as well.  It should be noted that \eqref{eqn:wyeloadI} and \eqref{eqn:deltaloadI} are different than the constant-current-phasor model employed in \cite{Garces2016}.

For future use, if $k=\texttt{Lin}(n,\phi)$ where $n \in \mc{N}_{Y}$ and $\phi \in \Omega_n$, we define $s_L^k=s_{L_n}^{\phi}$ and $i_L^k=i_{L_n}^\phi$.                                                                                                                                                                              For  $k=\texttt{Lin}(n,\phi)$ where $n \in \mc{N}_{\Delta}$ and $\phi \in \Omega_n$, if $r(\phi) \in \Omega_n$,  define  $s_L^k=s_{L_n}^{\phi r(\phi)}$,  $i_L^k=i_{L_n}^{\phi r(\phi)}$, and $\mb{e}_k=\mb{e}_n^{\phi r(\phi)}$.  If $r(\phi) \notin \Omega_n$, then $\Omega_n=\{\phi, \phi'\}$, and define $s_L^k=0$, $i_L^k=0$, and $\mb{e}_k=\mb{e}_n^{\phi \phi'}$.

\section{The Z-Bus Method}
\label{sec:zbus}
For a three-phase network, the multidimensional Ohm's law is given by
\begin{IEEEeqnarray}{rCl}
\bmat{ \mb{i} \\ \mb{i}_{\mr{S}} } = \bmat{ \mb{Y} & \mb{Y}_{\mr{NS}} \\ \mb{Y}_{\mr{SN}} & \mb{Y}_{\mr{SS}} }\bmat{\mb{v} \\ \mb{v}_{\mr{S}}},  \label{eqn:ohm}
\end{IEEEeqnarray}
where $\mb{i}_{\mr{S}}$ is the complex current injection of the slack bus.   Matrices $\mb{Y}$, $\mb{Y}_{\mr{NS}}$, $\mb{Y}_{\mr{SN}}$, and $\mb{Y}_{\mr{SS}}$ are formed by concatenating the admittance matrices  of transmission lines, transformers and voltage regulators, given e.g., in \cite{Chen1991}.
It follows from  \eqref{eqn:ohm} that
\begin{IEEEeqnarray}{rCl}
\mb{i} (\mb{v}) &=& \mb{Y} \mb{v} + \mb{Y}_{\mr{NS}} \mb{v}_{\mr{S}}, \label{eqn:iv}
\end{IEEEeqnarray}
where the dependence of injected currents on nodal voltages is shown explicitly. 

Using~\eqref{eqn:netI},  $\mb{i}(\mb{v})$ can be decomposed into three parts:
\begin{IEEEeqnarray}{rCl}
\mb{i}(\mb{v}) = \mb{i}_{\mr{PQ}} (\mb{v}) + \mb{i}_{\mr{I}}(\mb{v}) + \mb{i}_{\mr{Z}}(\mb{v}), \label{eqn:decomposeIv}
\end{IEEEeqnarray}
where $\mb{i}_{\mr{PQ}}(\mb{v})=\{\mb{i}_{\mr{PQ}_n}\}_{n \in \mc{N} \backslash \{\mr{S} \}}$,  $\mb{i}_{\mr{I}}(\mb{v})=\{\mb{i}_{\mr{I}_n}\}_{n \in \mc{N} \backslash \{\mr{S} \}}$, $\mb{i}_{\mr{Z}}(\mb{v})=\{\mb{i}_{\mr{Z}_n}\}_{n \in \mc{N} \backslash \{\mr{S} \}}$. Moreover, \eqref{eqn:wyeloadZ} and \eqref{eqn:deltaloadZ} reveal that  $\mb{i}_{\mr{Z}}(\mb{v})$ can be written as a linear function of $\mb{v}$ as follows:
\begin{IEEEeqnarray}{rCl}
\mb{i}_{\mr{Z}}(\mb{v}) = -  \mb{Y}_{\mr{L}} \mb{v},  \label{eqn:ivYloads}
\end{IEEEeqnarray}
where $\mb{Y}_{\mr{L}} \in \mbb{C}^{J \times J}$ has entries given by
\begin{subequations}
\label{eqngroup:formingY_L}
\begin{IEEEeqnarray*}{rClC}
\mb{Y}_{\mr{L}}(j,j) &=& y_{L_n}^{\phi}, \text{if } 
 j=\texttt{Lin}(n,\phi), n\in\mc{N}_Y, \label{eqn:formingY_Lyelements}\\
\mb{Y}_{\mr{L}}(j,j) &=& \sum\limits_{\phi ' \in \Omega_n \backslash \{\phi\}} y_{L_n}^{\phi\phi'}, \text{if },  j=\texttt{Lin}(n,\phi), n\in\mc{N}_\Delta, \label{eqn:formingY_Ldeltaelements1} \IEEEeqnarraynumspace \\
\mb{Y}_{\mr{L}}(j,k) &=& -y_{L_n}^{\phi \phi'}, j = \texttt{Lin}(n,\phi),  k=\texttt{Lin}(n,\phi'), n \in \mc{N}_\Delta,  \IEEEeqnarraynumspace \label{eqn:forminY_Ldeltaelements2} 
\end{IEEEeqnarray*}
\end{subequations}
and the remaining entries of $\mb{Y}_L$ are all zero. 

Substituting \eqref{eqn:decomposeIv} and \eqref{eqn:ivYloads} in \eqref{eqn:iv} and after some manipulation the following fixed-point equation for  $\mb{v}$ is obtained:
\begin{IEEEeqnarray}{rCl}
\mb{v}=\mb{Z}\left[ \mb{i}_{\mr{PQ}} (\mb{v}) +\mb{i}_{\mr{I}}(\mb{v})\right]+\mb{w},  \label{eqn:zbusv1}
\end{IEEEeqnarray}
where $\mb{Z}=(\mb{Y}+\mb{Y}_{\mr{L}})^{-1}$, and
\begin{IEEEeqnarray}{rCl}
\mb{w}= - \mb{Z}\mb{Y}_{\mr{NS}} \mb{v}_{\mr{S}}. \label{eqn:w}
\end{IEEEeqnarray}

Equation \eqref{eqn:zbusv1} lends itself to the well-known Z-Bus method for three-phase networks, outlined in e.g., \cite{Chen1991pf}:
\begin{IEEEeqnarray}{rCl}
\mb{v}[t+1]=\mb{Z}\left[ \mb{i}_{\mr{PQ}} (\mb{v}[t]) +\mb{i}_{\mr{I}}(\mb{v}[t])\right]+\mb{w},  \label{eqn:zbusv}
\end{IEEEeqnarray}
where $\mb{v}[t]$ is the value of $\mb{v}$ at iteration $t$. Notice that \eqref{eqn:zbusv} is an application of Picard's iteration for solving nonlinear systems of equations \cite[pp. 182]{OrtegaRheinBoldt1970}.

\section{Sufficient conditions for convergence of the Z-{Bus} method to a unique solution}
\label{sec:conditions}
In this section,   sufficient conditions are presented such that the iteration of the Z-Bus method~\eqref{eqn:zbusv} is \textit{contracting}, i.e., the iteration  converges geometrically to a unique solution of~\eqref{eqn:zbusv1}.  

To present these conditions, it is beneficial to make a change of variables in \eqref{eqn:zbusv}. Specifically, let $\mb{u} = \bs{\Lambda} ^{-1}\mb{v}[t]$ where $\bs{\Lambda} \in \mbb{C}^{J \times J}$ is an invertible diagonal matrix with entries $\lambda_j$.   Matrix $\mb{\Lambda}$ will serve as a design parameter which can be leveraged to potentially  expand the convergence region of the Z-Bus method. Parameterizing $\mb{v}$ by $\mb{u}$ and multiplying both sides of \eqref{eqn:zbusv} by $\bs{\Lambda}^{-1}$ yield the following iteration for $\mb{u}$:
 \begin{IEEEeqnarray}{rCl}
\mb{u}[t+1]=\mb{T}(\mb{u}[t]) \label{eqn:zbusu}, 
\end{IEEEeqnarray}
where $\mb{T}(\mb{u}): \mbb{C}^J \rightarrow \mbb{C}^J$ is the following mapping
\begin{IEEEeqnarray}{rCl}
\mb{T}(\mb{u})=\bs{\Lambda}^{-1} \mb{Z}\left[\mb{i}_{\mr{PQ}} ( \bs{\Lambda} \mb{u}) + \mb{i}_{\mr{I}}(\bs{\Lambda} \mb{u}) \right] + \bs{\Lambda}^{-1}\mb{w},  \label{eqn:Tu}\IEEEeqnarraynumspace
\end{IEEEeqnarray}
and the Z-Bus method of \eqref{eqn:zbusv} can be equivalently written as
\begin{IEEEeqnarray}{rCl}
\mb{v}[t+1]=\mb{\Lambda}\mb{T}(\mb{\Lambda}^{-1}\mb{v}[t]) \label{eqn:zbusvT}. 
\end{IEEEeqnarray}
Since the iterations in \eqref{eqn:zbusu} and \eqref{eqn:zbusvT} have a one-to-one correspondence, convergence of \eqref{eqn:zbusu} to a fixed point $\mb{u}^{\mr{fp}}=\mb{T}(\mb{u}^{\mr{fp}})$ is equivalent to the  convergence of \eqref{eqn:zbusvT} to a respective solution $\mb{v}^{\mr{fp}}= \mb{\Lambda}\mb{u}^{\mr{fp}}$ of~\eqref{eqn:zbusv1}.  

In what follows, conditions such that the mapping  $\mb{T}(\mb{u})$ in \eqref{eqn:Tu} is a contraction over a region $\mb{D} \subseteq \mbb{C}^J$ are derived.   As a consequence, $\mb{T}(\mb{u})$ has a unique fixed point over $\mb{D}$, and iterations \eqref{eqn:zbusu} converge geometrically to the unique fixed point \cite[Prop. 3.1.1]{BertsekasTsitsiklis1989}.  The definition of contraction is as follows:
\begin{definition*}
For a closed set $\mb{D} \subseteq \mbb{C}^J$, a mapping  $\mb{T}(\mb{u})$ is  a contraction  on $\mb{D}$  if the following two properties hold:
\begin{enumerate}
\item \textit{Self-mapping property:}   $ \mb{u} \in \mb{D}  \Rightarrow \mb{T}(\mb{u}) \in \mb{D}$.
\item \textit{Contraction property:}  $\mb{u}, \tilde{\mb{u}} \in \mb{D} \Rightarrow \|\mb{T}(\mb{u}) - \mb{T}(\tilde{\mb{u}}) \| \leq \alpha \|\mb{u} - \tilde{\mb{u}}\|$, 
\end{enumerate}
where $\|.\|$ is some norm and $\alpha \in [0,1)$ is a constant.
\end{definition*}
In regards to the previous definition, any number $\alpha \in [0,1)$ that upperbounds the ratio $\frac{\|\mb{T}(\mb{u})- \mb{T}(\tilde{\mb{u})}\|}{\| \mb{u}- \tilde{\mb{u}}\|}$ uniformly for all $\mb{u}, \tilde{\mb{u}} \in  \mb{D}$ ($\mb{u} \neq \tilde{\mb{u}}$), is called \emph{contraction modulus}.

In order to state the conditions for~\eqref{eqn:Tu} to be a contraction map, the following quantities are defined:
 \begin{subequations}
 \label{eqnGroup:feederQuantities}
 \begin{IEEEeqnarray}{rCl}
 \IEEEeqnarraymulticol{3}{l}{\underline{w}=\min_{k \in \mc{J}} |w_k| , \bar{\lambda} = \max_{k \in \mc{J}} |\lambda_k|,  \underline{\rho}=\min_{k \in \mc{J}_{\Delta}} \{|\mb{e}_{k}^T \mb{w}_{\texttt{Node}[k]}|\}, } \label{eqnGroup:knownminimums}\IEEEeqnarraynumspace \\
\mb{s}^Y&=& \{s_{L}^k\}_{k \in \mc{J}_Y}  \in \mbb{C}^{|\mc{J}_Y|}, \mb{i}^Y = \{i_{L}^k\}_{k \in \mc{J}_Y} \in \mbb{C}^{|\mc{J}_Y|}, \IEEEeqnarraynumspace \\
\mb{w}^Y &=& \{w_k\}_{k \in \mc{J}_Y} \in \mbb{C}^{|\mc{J}_Y|},
\mb{Z}^Y = [ \mb{Z}_{\bullet k} ]_{k \in \mc{J}_{Y}} \mbb{C}^{J \times |\mc{J}_Y|}, \IEEEeqnarraynumspace\label{eqn:zy}   \\
\mb{\Lambda}^Y &=& \diag(\{\lambda_k\}_{k \in \mc{J}_Y}) \in \mbb{C}^{|\mc{J}_Y| \times |\mc{J}_Y|},\label{eqn:lambday} \\
\mb{s}^\Delta &=& \{\mb{s}_{L}^k\}_{k \in \mc{J}_\Delta} \in \mbb{C}^{|\mc{J}_\Delta|},   
\mb{i}^\Delta = \{\mb{i}_{L}^k\}_{k \in \mc{J}_\Delta} \in \mbb{C}^{|\mc{J}_\Delta|},  \label{eqn:idelta} \\
\mb{w}^{\Delta} &=&\{ \mb{e}_k^T \mb{w}_{\texttt{Node}[k]} \}_{k \in \mc{J}_{\Delta}}\in \mbb{C}^{|\mc{J}_\Delta|}, \label{eqn:wdelta}\\
\mb{Z}^\Delta &=& [ \mb{Z}_{\bullet l_k}^{\Delta} ]_{k \in \mc{J}_{\Delta}}  \in \mbb{C}^{J \times |\mc{J}_{\Delta}|}, \label{eqn:zdelta}\\
\IEEEeqnarraymulticol{3}{l}{ \text{ where } l_k \in \{1,\ldots, |\mc{J}_{\Delta}| \} \text{ is the order of $k$ in $\mc{J}_{\Delta}$},  \text{ and } } \notag \\
&&  \mb{Z}_{\bullet l_k}^{\Delta}= \notag  \left\{\begin{IEEEeqnarraybox}[
\IEEEeqnarraystrutmode
\IEEEeqnarraystrutsizeadd{0pt}
{0pt}
][c]{lCl}
\mb{Z}_{\bullet k} - \mb{Z}_{\bullet k'},& &\: \text{if }  \phi, r(\phi) \in \Omega_n, \IEEEeqnarraynumspace\\
\mb{0}_J,& & \: \text{if } \phi \in \Omega_n, r(\phi) \notin \Omega_n, \IEEEeqnarraynumspace
\end{IEEEeqnarraybox}
\, \right. \\
&& \text{with } k=\texttt{Lin}(n,\phi), k'=\texttt{Lin}(n,r(\phi)),   \notag  \\
\mb{\Lambda}^\Delta &=&\diag(\{ \max\nolimits_{l \in \mc{J}_{\texttt{Node}[k]}} |\lambda_l |  \}_{ k \in \mc{J}_{\Delta}}) \in \mbb{C}^{ |\mc{J}_{\Delta}| \times |\mc{J}_{\Delta}|}. \label{eqn:lambdadelta}  \IEEEeqnarraynumspace 
\end{IEEEeqnarray}
\end{subequations}
In  \eqref{eqnGroup:feederQuantities}, the notation $\mb{w}_\mb{n}=\{w_k\}_{k \in \mc{J}_n}$ relates the indices in $\mb{w}$ to corresponding ones in $\mb{v}$.  Respectively for wye and delta connections, quantities $\mb{s}^Y$ and $\mb{s}^{\Delta}$ collect all constant-power loads,  $\mb{i}^Y$ and $\mb{i}^{\Delta}$ collect all constant-current loads, $\mb{Z}^Y$ and $\mb{Z}^{\Delta}$  collect corresponding columns of the impedance matrix $\mb{Z}$, and  $\mb{w}^Y$, $\mb{w}^{\Delta}$  collect the voltages induced by the slack bus. Matrix $\mb{\Lambda}^Y$ contains the elements of the design matrix $\mb{\Lambda}$ that correspond to wye nodes and matrix $\mb{\Lambda}^{\Delta}$ selects the maximum of the per node  entries of the design matrix $\mb{\Lambda}$ for delta nodes.

Theorem \ref{thm:main} establishes the  convergence of the Z-Bus method in three-phase distribution networks with ZIP loads, and is proved in Appendix \ref{sec:appendixProof} using intermediate results in Appendix~\ref{sec:app:voltagebounds}.

\begin{theorem}
\label{thm:main}
Define the closed ball  $\mb{D}_{R}:=\{\mb{u} \in \mbb{C}^J\ | \  \|\mb{u} - \bs{\Lambda}^{-1}\mb{w}  \|_{\infty} \le R\}$ where $R>0$.  
Then $\mb{T}(\mb{u})$ in \eqref{eqn:Tu} is a contraction mapping on $\mb{D}_R$ with contraction modulus $\alpha$ if the following four conditions hold:
\begin{equation}
\begin{IEEEeqnarraybox}{rCl}
1-R \bar{\lambda} / \underline{w} > 0, 
\end{IEEEeqnarraybox} \tag{C1} \label{eqn:linetoneutralcondition}
\end{equation}
\begin{equation}
\begin{IEEEeqnarraybox}{rCl}
1-2R \bar{\lambda} / \underline{\rho} > 0, 
\end{IEEEeqnarraybox} \tag{C2} \label{eqn:linetolinecondition}
\end{equation}
\begin{equation}
   \begin{IEEEeqnarraybox}{rCl}
&&\frac{\| \bs{\Lambda}^{-1}  \mb{Z}^Y \diag(\mb{s}^Y) \diag(\mb{w}^Y)^{-1} \|_{\infty}}{1- R \bar{\lambda} / \underline{w}} \notag  \IEEEeqnarraynumspace\\
&&+ \: \frac{\| \bs{\Lambda}^{-1}  \mb{Z}^\Delta \diag(\mb{s}^\Delta) \diag(\mb{w}^\Delta)^{-1} \|_{\infty}}{1  - 2R  \bar{\lambda}/\underline{\rho}}  \IEEEeqnarraynumspace  \notag \\
&&+\: \|  \bs{\Lambda}^{-1}\mb{Z}^Y \diag(\mb{i}^Y) \|_{\infty} + \|  \bs{\Lambda}^{-1} \mb{Z}^{\Delta} \diag(\mb{i}^\Delta) \|_{\infty} \le R  \hspace{-0.5cm} \IEEEeqnarraynumspace
\end{IEEEeqnarraybox} \tag{C3} \label{eqn:selfmappingcondition}
\end{equation}
\begin{equation}
\begin{IEEEeqnarraybox}{rCl}
\IEEEeqnarraymulticol{3}{l}{  \frac{ \|  \mb{\Lambda}^{-1} \mb{Z}^Y \diag(\mb{s}^Y) \mb{\Lambda}^Y [\diag(\mb{w}^Y)^{-1}]^2 \|_{\infty}}{ (1 - R \bar{\lambda}/ \underline{w})^2}} \notag  \IEEEeqnarraynumspace \\
 \IEEEeqnarraymulticol{3}{l}{+ \: \frac{2 \| \mb{\Lambda}^{-1} \mb{Z}^{\Delta} \diag(\mb{s}^\Delta) \mb{\Lambda}^{\Delta} [\diag(\mb{w}^\Delta)^{-1}]^2 \|_{\infty}} { (1 - 2R \bar{\lambda} / \underline{\rho})^2}} \IEEEeqnarraynumspace  \notag \\
 \IEEEeqnarraymulticol{3}{l}{+\: \frac{2 \| \mb{\Lambda}^{-1} \mb{Z}^Y \diag(\mb{i}^Y) \mb{\Lambda}^Y \diag(\mb{w}^Y)^{-1}\|_{\infty}}{1- R \bar{\lambda}/ \underline{w}}} \notag \IEEEeqnarraynumspace \\
  \IEEEeqnarraymulticol{3}{l}{+ \: \frac{ 4\| \mb{\Lambda}^{-1} \mb{Z}^{\Delta} \diag(\mb{i}^\Delta) \mb{\Lambda}^{\Delta} \diag(\mb{w}^\Delta)^{-1}  \|_{\infty}} { 1 - 2R \bar{\lambda} / \underline{\rho}} \le \alpha <1. \hspace{-0.5cm}}      \IEEEeqnarraynumspace 
 \end{IEEEeqnarraybox}
 \tag{C4} \label{eqn:contractioncondition} 
\end{equation} 
\end{theorem}

If conditions \eqref{eqn:linetoneutralcondition}--\eqref{eqn:contractioncondition} hold for some $R >0$, then $\mb{T}(\mb{u})$ is a contraction mapping and has as a consequence a unique fixed point $\mb{u}^{\mr{fp}}$ in the ball defined by $\mb{D}_R$. In other words, the power flow equations have a unique solution in $\mb{D}_R$.
The value of $R$ can be interpreted as the infinity norm of the voltage difference from the  no-load power flow solution scaled by $\mb{\Lambda}$. The latter is obtained as  $\mb{u}= \bs{\Lambda}^{-1} \mb{w}$ upon setting  $\mb{i}_{\mr{PQ}}=\mb{i}_{\mr{I}}=\mb{0}$ in~\eqref{eqn:Tu}.  

Moreover, the contraction property of $\mb{T}(\mb{u})$ implies the following relationships for successive iterates of the Z-Bus method upon initialization with $\mb{u}[0] \in \mb{D}_R$:
\begin{subequations}
\begin{IEEEeqnarray}{rCl}
\| \mb{u}[t+1] - \mb{u}^{\mr{fp}} \|_{\infty} &\le& \alpha \| \mb{u}[t] - \mb{u}^{\mr{fp}} \|_{\infty}, \label{eqn:utufp} \\
\| \mb{u}[t+1] - \mb{u}[t]\|_{\infty} &\le& \alpha \| \mb{u}[t] - \mb{u}[t-1] \|_{\infty}. \label{eqn:numericalalpha}
\end{IEEEeqnarray}
\end{subequations}
Equation \eqref{eqn:utufp} implies that the iterates in \eqref{eqn:zbusu} converge to $\mb{u}^{\mr{fp}}$ with a  convergence rate that is upperbounded by $\alpha$.  Equation \eqref{eqn:numericalalpha} can be used to numerically evaluate the convergence rate by computing the ratio $\frac{ \|\mb{u}[t+1]-\mb{u}[t]\|_{\infty}}{\| \mb{u}[t]-\mb{u}[t-1]\|_{\infty}}$.  We will see in Section \ref{sec:numericaltests} for the IEEE test feeders  that  the aforementioned ratio turns out to be much smaller than the estimate provided by the left-hand side of \eqref{eqn:contractioncondition}.

Theorem \ref{thm:main} on $\mb{T}(\mb{u})$ implies certain convergence properties for iteration~\eqref{eqn:zbusv}, as stated next.
\begin{theorem}
\label{thm:second}
If conditions \eqref{eqn:linetoneutralcondition}--\eqref{eqn:contractioncondition} hold,  then by initializing $\mb{v}[0] \in \mb{D}'_R := \{ \mb{v}| \: \|\mb{\Lambda}^{-1} (\mb{v}-\mb{w})\|_{\infty} \le R \}$, the Z-Bus iterations in \eqref{eqn:zbusv} remain within $\mb{D}'_R$ and converge to a unique solution. Moreover, the convergence is R-linear with rate $\alpha$ [i.e., the sequence generated by~\eqref{eqn:zbusv} is dominated by the geometric sequence $B\alpha^t$ where $B$ is a non-negative scalar].
\end{theorem}
\begin{proof}
The self-mapping property of $\mb{T}(\mb{u})$ and the one-to-one correspondence between $\mb{v}$ and $\mb{u}$ guarantee that the iterates $\mb{v}[t]$ remain within $\mb{D}'_R$.  Regarding the convergence rate of \eqref{eqn:zbusv}, note that the contraction property of $\mb{T}(\mb{u})$ implies 
(cf.~\cite[Prop. 3.1.1]{BertsekasTsitsiklis1989})
\begin{IEEEeqnarray}{rCl}
\| \mb{u}[t]-\mb{u}^{\mr{fp}}\|_{\infty} \le \alpha^t \|\mb{u}[0]-\mb{u}^{\mr{fp}}\|_{\infty}. \label{eqn:ucontractionimplication}
\end{IEEEeqnarray}
Since we have $\mb{u}= \mb{\Lambda}^{-1} \mb{v}$, it follows from~\eqref{eqn:ucontractionimplication} that
\begin{IEEEeqnarray}{rCl}
\| \mb{\Lambda}^{-1} ( \mb{v}[t] - \mb{v}^{\mr{fp}} ) \|_{\infty} \le \alpha^t \| \mb{\Lambda}^{-1}(\mb{v}[0] - \mb{v}^{\mr{fp}} ) \|_{\infty}. \label{eqn:vcontraction1}
\end{IEEEeqnarray}
From the definition of the infinity norm it follows that
\begin{IEEEeqnarray}{lll}
&\| \mb{\Lambda}^{-1} ( \mb{v}[t] - \mb{v}^{\mr{fp}} ) \|_{\infty} &= \max_{j \in \mc{J}}  \left| \frac{1}{\lambda_j} (v_j[t] - v_j^{\mr{fp}} )\right|  \notag  \\
&&= \max_{j \in \mc{J}} \left\{ \left| \frac{1}{\lambda_j}\right| |v_j[t] - v_j^{\mr{fp}} |\right\} \notag  \\
&&\ge   \frac{1}{\bar{\lambda}} \max_{j \in \mc{J}}  |v_j[t] - v_j^{\mr{fp}} | \notag \\
\Rightarrow & \| \mb{\Lambda}^{-1} ( \mb{v}[t] - \mb{v}^{\mr{fp}} ) \|_{\infty}   &\ge  \frac{1}{\bar{\lambda}} \| \mb{v}[t] - \mb{v}^{\mr{fp}} \|_{\infty}. \IEEEeqnarraynumspace\label{eqn:vcontraction1upper}
\end{IEEEeqnarray}
 Moreover, we have that
\begin{IEEEeqnarray}{lll}
&\| \mb{\Lambda}^{-1} ( \mb{v}[0] - \mb{v}^{\mr{fp}} ) \|_{\infty} &= \max_{j \in \mc{J}}  \left| \frac{1}{\lambda_j} (v_j[0] - v_j^{\mr{fp}} )\right| \notag  \\
&&\le  \max_{j \in \mc{J}} \left| \frac{1}{\min_{k \in \mc{J}} | \lambda_k |} (v_j[0] - v_j^{\mr{fp}} )\right| \notag  \\
&&\le \frac{1}{\min_{k \in \mc{J}} | \lambda_k |}  \max_{j \in \mc{J}}   |v_j[0] - v_j^{\mr{fp}} | \notag \\
\Rightarrow &\| \mb{\Lambda}^{-1} ( \mb{v}[0] - \mb{v}^{\mr{fp}} ) \|_{\infty}  & \le  \frac{1}{\min_{k \in \mc{J}} | \lambda_k |}  \| \mb{v}[0] - \mb{v}^{\mr{fp}} \|_{\infty}. \label{eqn:vcontraction1lower}
\end{IEEEeqnarray}

Combining \eqref{eqn:vcontraction1upper}, \eqref{eqn:vcontraction1lower}, and \eqref{eqn:vcontraction1} yields
\begin{IEEEeqnarray}{rCl}
\|\mb{v}[t]-\mb{v}^{\mr{fp}}\|_{\infty} \le B \alpha^t, \label{eqn:vcontractionimplication}
\end{IEEEeqnarray}
where $B=\frac{\bar{\lambda}}{\min_{j \in \mc{J}} |\lambda_j|} \|\mb{v}[0]-\mb{v}^{\mr{fp}}\|_{\infty} \ge 0$.  Since the quantity $\| \mb{v}[t] - \mb{v}^{\mr{fp}}\|_{\infty}$ is upperbounded by a geometric sequence, the rate of convergence is (informally) said to be geometric; formally, the convergence is R-linear with rate $\alpha$ \cite[pp. 619]{NocedalWright2006}. 
\end{proof}

Conditions \eqref{eqn:linetoneutralcondition}--\eqref{eqn:contractioncondition} can be expressed via inequalities involving polynomials of degree at most four in the radius of the contraction region. The coefficients of the polynomials can be readily written in terms of the network parameters and loads. Quartic polynomials and their roots have been completely characterized, and can be routinely computed in terms of the polynomial coefficients---see for example \cite{Wolfram1}. The convergence region can thus be easily computed. Significant analytical advantages can also be afforded by the presented conditions, as the convergence region can potentially be expressed explicitly in terms of the network loads and parameters. An alternative set of nonlinear non-polynomial conditions for convergence of the Z-Bus method is presented in \cite{BazrafshanGatsis2016}. 

Conditions \eqref{eqn:linetoneutralcondition}--\eqref{eqn:contractioncondition} can also be graphically solved, as explained next.  First,  the quantities in \eqref{eqnGroup:feederQuantities} must be computed which depend only on the feeder data.   The set of $R$ values satisfying each condition can simply be plotted on the real line.  The intersection of the feasible regions in these plots gives a range for $R$.  
In the next section,  the valid range of $R$ satisfying conditions \eqref{eqn:linetoneutralcondition}--\eqref{eqn:contractioncondition} is computed for IEEE distribution test feeders, and the convergence of the Z-Bus method is numerically illustrated.

\section{Numerical verification on IEEE feeders}
\label{sec:numericaltests}
In this section,  the range of $R$ where conditions \eqref{eqn:linetoneutralcondition}--\eqref{eqn:contractioncondition} are satisfied is given for  IEEE-13, IEEE-37, and IEEE-123  distribution test feeders \cite{RadialFeeders2001}.    Furthermore,  the self-mapping property as well as the convergence of the Z-Bus method are numerically confirmed for iterations on voltages in \eqref{eqn:zbusv}.  This is because the Z-Bus method is routinely implemented on the variable $\mb{v}$ as opposed to $\mb{u}$. 

The feeder data are available online \cite{feederdata}.  We model transformer nodal admittances according to \cite{Chen1991pf}.  However, to avoid singularities in the bus admittance matrix,  delta-delta transformers are modeled using a modification suggested in \cite{Gorman1992} which entails connecting a small shunt admittance to the ground.   For voltage regulators, the models are derived from \cite[Ch. 7]{KerstingBook2001}.\footnote{The MATLAB scripts that compute the convergence regions for the  IEEE test feeders and run  the Z-Bus method are provided online at the following page: \url{https://github.com/hafezbazrafshan/contraction-mapping.}}

To numerically verify conditions \eqref{eqn:linetoneutralcondition}--\eqref{eqn:contractioncondition},  the design parameter $\mb{\Lambda}$ is selected to be $\mb{\Lambda}=\diag(\mb{w})$.   This  choice allows for the center of $\mb{D}_R$ to be the  vector of all one's. This practically means that the no-load solution for scaled voltages is $\mb{u}= \mb{\Lambda}^{-1}\mb{w}=\mb{1}$.

For the IEEE-13 test feeder, conditions \eqref{eqn:linetoneutralcondition}--\eqref{eqn:contractioncondition} yield the following inequalities for $R$:
\begin{subequations}
\label{eqnGroup:ieee13conditionsLAMBDAW}
\begin{IEEEeqnarray}{rCl}
\IEEEeqnarraymulticol{3}{l}{1-  1.043 R > 0\text{;} \quad 1- 1.203 R  > 0} \IEEEeqnarraynumspace \label{eqn:ieee13f1LAMBDAW}  \IEEEeqnarraynumspace \\
\IEEEeqnarraymulticol{3}{l}{\frac{0.127}{1- 1.043 R} +  \frac{  0.041}{1-1.203 R	} +  0.032+   + 0.01\le R} \label{eqn:ieee13f3LAMBDAW}  \IEEEeqnarraynumspace\\
\frac{0.127}{(1- 1.043 R)^2} &+& \frac{0.048}{(1-1.203 R)^2} \notag \\
&+&   \frac{0.064}{1- 1.043 R}+ \frac{0.024
}{1-1.203 R}  <1. \label{eqn:ieee13f4LAMBDAW} \IEEEeqnarraynumspace
\end{IEEEeqnarray}
\end{subequations}
For the IEEE-37 test feeder, the conditions are 
\begin{subequations}
\label{eqnGroup:ieee37conditionsLAMBDAW}
\begin{IEEEeqnarray}{rCl}
\IEEEeqnarraymulticol{3}{l}{1- 1.037R >0\text{;}  \quad 1-1.198 R >0} \label{eqn:ieee37f1LAMBDAW} \IEEEeqnarraynumspace \\
\IEEEeqnarraymulticol{3}{l}{ \frac{  0.037}{1-1.198 R} + 0.018 \le R} \label{eqn:ieee37f3LAMBDAW} \IEEEeqnarraynumspace\\
 \frac{0.043}{(1-1.198 R)^2} &+& \frac{0.042
}{1-1.198 R}  <1. \label{eqn:ieee37f4LAMBDAW} \IEEEeqnarraynumspace
\end{IEEEeqnarray}
\end{subequations}
Finally, for the IEEE-123 feeder, the conditions are
\begin{subequations}
\label{eqnGroup:ieee123conditionsLAMBDAW}
\begin{IEEEeqnarray}{rCl}
\IEEEeqnarraymulticol{3}{l}{1-  1.078R >0\text{;} \quad 1-   1.238 R >0} \label{eqn:ieee123f1LAMBDAW} \\
\IEEEeqnarraymulticol{3}{l}{\: \frac{ 0.129}{ 1 -1.078R} + \frac{ 0.001}{1-  1.238  R	} + 0.030   + 0.012  \le R} \label{eqn:ieee123f3LAMBDAW} \IEEEeqnarraynumspace \\
 \frac{0.129
}{(1 -1.0777R)^2} &+&\frac{0.001}{(1-1.238 R)^2} \notag \\
&+& \frac{0.060}{1-1.078 R}+\frac{0.027}{1-1.238R}  <1.\label{eqn:ieee123f4LAMBDAW} \IEEEeqnarraynumspace
\end{IEEEeqnarray}
\end{subequations}

	For each feeder, the subsets of the real line where conditions \eqref{eqn:linetoneutralcondition}--\eqref{eqn:contractioncondition} are satisfied are depicted in Fig.~\ref{fig:regions}, by solving \eqref{eqnGroup:ieee13conditionsLAMBDAW}, \eqref{eqnGroup:ieee37conditionsLAMBDAW}, and \eqref{eqnGroup:ieee123conditionsLAMBDAW}.  The intersection of these regions for the IEEE-13, IEEE-37 and IEEE-123 feeders is   $[R_{\min}, R_{\max}]=[0.29,0.48]$, $[R_{\min}, R_{\max}]=[0.06,0.64]$, and $[R_{\min}, R_{\max}]=[0.22,0.54]$, respectively.  Both  $R_{\min}$ and $R_{\max}$  reveal important information about the Z-Bus convergence.  Specifically, larger $R_{\max}$ guarantees convergence when the initialization is far from the unique solution.  Smaller $R_{\min}$ guarantees that the unique solution is close to the center of the ball.  The latter implies that the deviation of the load-flow voltage solution from the no-load solution is tightly bounded.  Notice that these regions are dependent on the design matrix $\mb{\Lambda} = \diag(\mb{w})$. For example,  setting $\mb{\Lambda}=\mb{I}_J$ for the IEEE-13 feeder yields $R_{\max}=0.60$ and also shifts the center of the ball from $\mb{1}$ to $\mb{w}$.  A study on  the impact of $\mb{\Lambda}$ on the contracting region remains for future work.
	
\begin{figure}[t]
\centering
\subfloat[]{\includegraphics[scale=0.5]{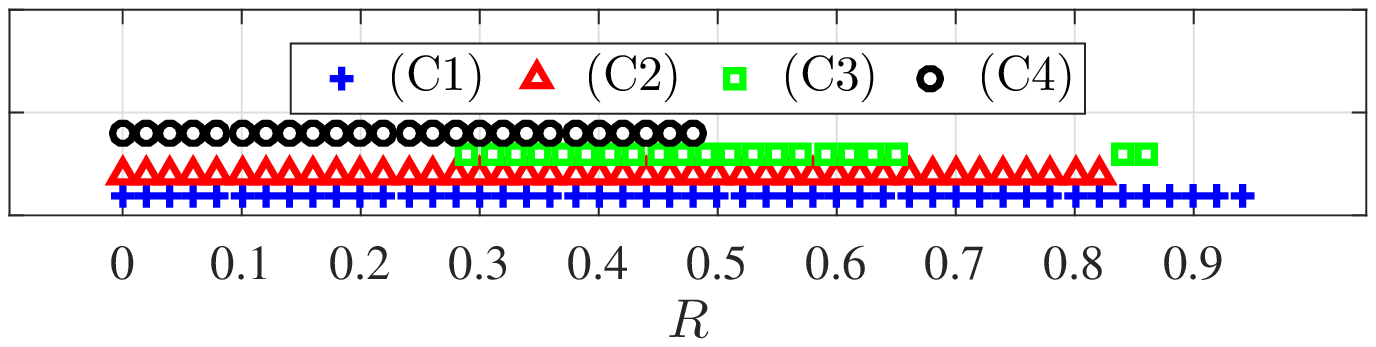}\label{fig:ieee13regions}}{} 
\vfill
\subfloat[]{\includegraphics[scale=0.5]{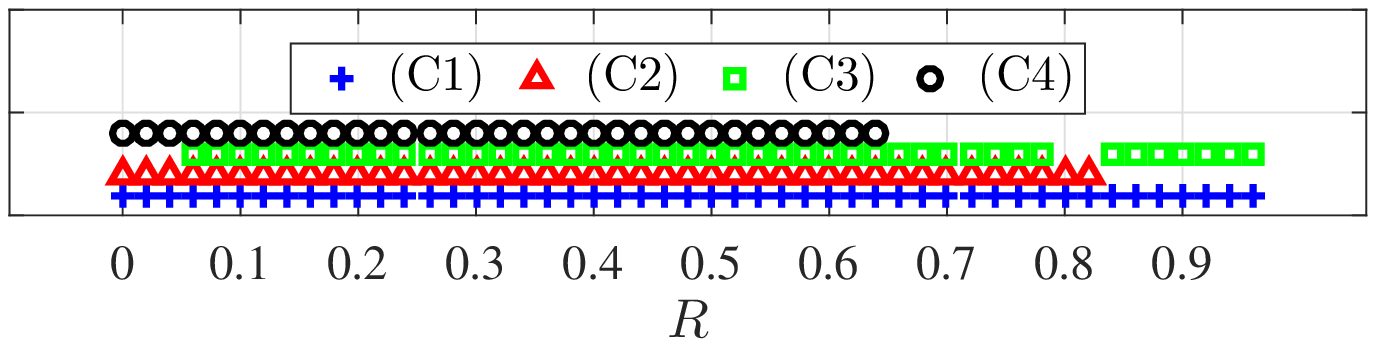}\label{fig:ieee37regions}}{}
\vfill
\subfloat[]{\includegraphics[scale=0.5]{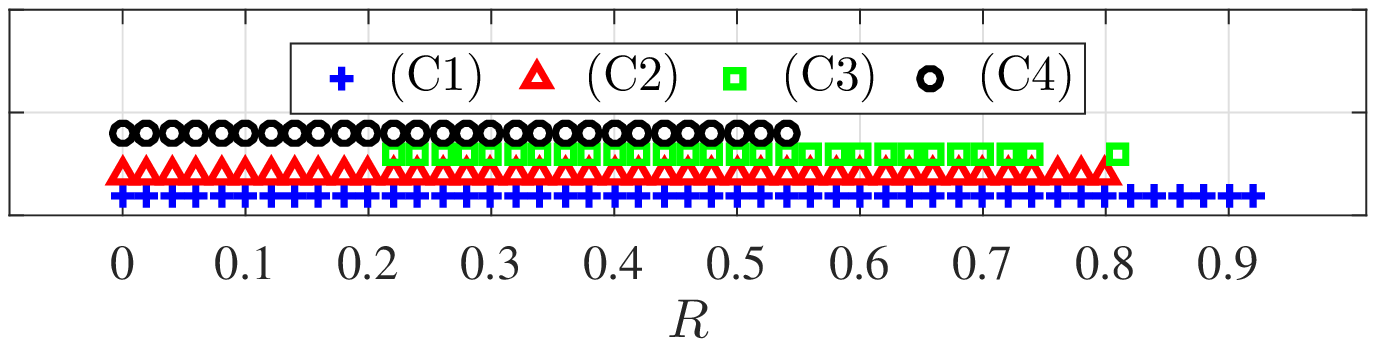}\label{fig:ieee123regions}}{}
\caption{\small Regions on the real line where conditions \eqref{eqn:linetoneutralcondition}--\eqref{eqn:contractioncondition} hold for  \protect\subref{fig:ieee13regions} IEEE-13, \protect\subref{fig:ieee37regions}  IEEE-37, and \protect\subref{fig:ieee123regions} IEEE-123 test feeders. The presence of a mark indicates that the condition is satisfied.}
\label{fig:regions}
\end{figure}

Figure~\ref{fig:contractionModulus} depicts the lower bound of the contraction modulus $\alpha$ evaluated from~\eqref{eqn:contractioncondition}  as a function of the feasible $R$  for the three IEEE feeders. It is revealed in  Fig.~\ref{fig:contractionModulus} that a more localized initialization (i.e., smaller $R$) yields smaller $\alpha$, with the smallest $\alpha$ occurring at $R_{\min}$. For the IEEE-13, IEEE-37, and IEEE-123 feeders, the smallest  $\alpha$ is respectively $0.50$, $0.09$, and $0.34$.  Upon initializing within $\mb{D}_{R_{\min}}$, the convergence rate is guaranteed to be at most as much as the aforementioned values.
\begin{figure}[t]
	\centering
\includegraphics[scale=0.45]{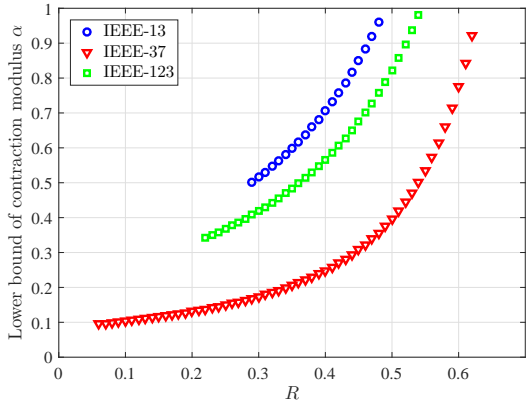}	
	\caption{  Lower bound of  contraction modulus $\alpha$ versus $R$ in the range for which conditions \eqref{eqn:linetoneutralcondition}--\eqref{eqn:contractioncondition} are satisfied.}
	\label{fig:contractionModulus}
\end{figure}

\begin{figure}[t]
	\centering
\includegraphics[scale=0.45]{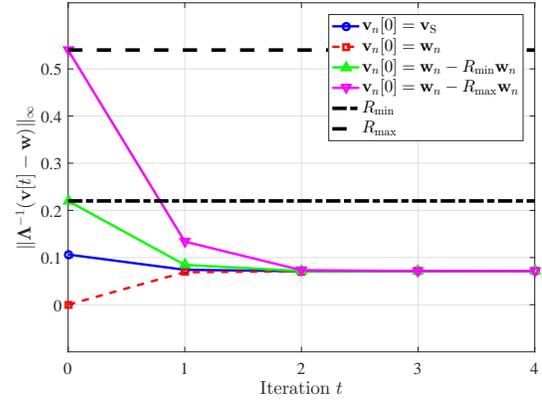}	
	\caption{ \small Demonstration of the self-mapping property of the Z-Bus iterations in \eqref{eqn:zbusv} for IEEE-123 test feeder.  As long as the Z-Bus method is initialized with $\mb{v}[0] \in \mb{D}'_R$, where $R$ satisfies conditions \eqref{eqn:linetoneutralcondition}--\eqref{eqn:contractioncondition}, successive iterations $\mb{v}[t]$ remain within $\mb{D}'_R$.}
	\label{fig:selfmapping}
\end{figure}

The self-mapping property of the Z-Bus iterations \eqref{eqn:zbusv} is demonstrated in Fig.~\ref{fig:selfmapping} for the IEEE-123 test feeder.  It is shown that when the initialization $\mb{v}[0]$ is within $\mb{D}'_R$, where $R \in [R_{\min}, R_{\max}]$, the sequence $\mb{v}[t]$ produced by  the Z-Bus method in \eqref{eqn:zbusvT} remains within $\mb{D}'_R$. In Fig.~\ref{fig:selfmapping}, a case of special interest is the graph corresponding to the initialization with the flat voltage profile, that is, when $\mb{v}_n[0]=\mb{v}_{\mr{S}}$, where $\mb{v}_{\mr{S}}=\{1,e^{-j2\frac{\pi}{3}},e^{j2\frac{\pi}{3}} \}$ is the voltage at the slack bus. The graph  shows that the typical initialization with the flat voltage profile  is within the contracting region; and upon initializing with $\mb{v}_n[0]=\mb{v}_{\mr{S}} \in \mb{D}_R'$ for all nodes, the iterates remain within $\mb{D}_R'$.

The distance between two successive iterates of the Z-Bus method for the IEEE-123 test feeder is plotted in Fig.~\ref{fig:zbusiterations}. It is shown that the distance $\| \mb{v}[t+1]  - \mb{v}[t]\|_{\infty}$ practically converges to zero within 4 iterations, implying that the fixed-point solution is obtained. For the sequences of Fig.~\ref{fig:zbusiterations}, the empirical convergence rate is also numerically calculated using \eqref{eqn:numericalalpha} and is depicted in Fig.~\ref{fig:IEEE123alpha}.  All the numerical convergence rates are smaller than $0.34$, which is the smallest theoretically obtained contraction modulus for the IEEE-123 network.
\begin{figure}[t]
\centering
\includegraphics[scale=0.45]{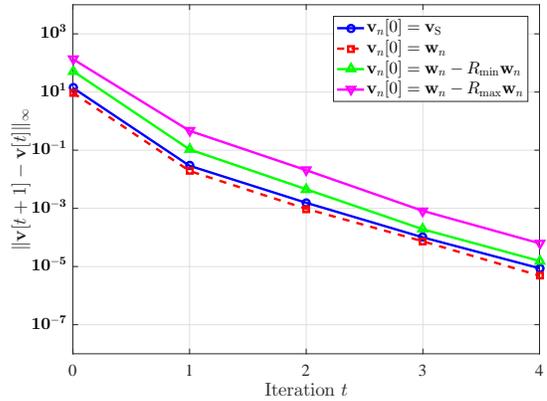} 	
\caption{\small Convergence of the Z-Bus iterations in \eqref{eqn:zbusv} for  IEEE-123 test feeder.  It is observed that  the Z-Bus method converges to the unique solution very fast---within 4 iterations.
	\label{fig:zbusiterations}}
\end{figure}

\begin{figure}[t]
\centering
\includegraphics[scale=0.45]{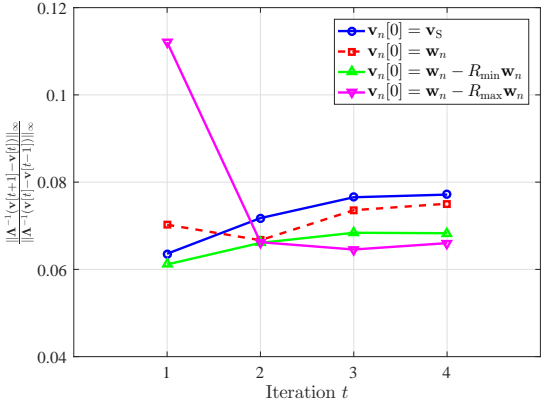} 	
\caption{ \small Numerical convergence rate of the Z-Bus iterations in \eqref{eqn:zbusv} for  the iterations in Fig.~\ref{fig:zbusiterations}.  The numerical convergence rate  at iteration $t$ is calculated using \eqref{eqn:numericalalpha}.  Notice that these convergence rates are all smaller than $0.34$, which is the tightest upper bound given in Fig.~\ref{fig:contractionModulus} for IEEE-123. 
	\label{fig:IEEE123alpha}}
\end{figure}

The merits of this work is that it characterizes certain regions with unique power flow solution in three-phase distribution networks and that it guarantees convergence of the Z-Bus method.  Working with networks that satisfy solvability conditions can provide a degree of assurance that the numerical algorithm---the Z-Bus method in this case---will converge to a solution. 

Further research is  required to obtain conditions that accurately characterize the largest region with a unique power flow solution.  As was previously alluded to, it is possible to leverage the design parameter $\mb{\Lambda}$ for this purpose. By setting $\mb{\Lambda}= \mr{diag}(\mb{w})$, the conditions are satisfied for the practical IEEE test networks; developing comprehensive methods for the selection of parameter $\mb{\Lambda}$ are left for future work.

It should be noted at this point that this work derives \emph{sufficient} conditions for convergence. Generally speaking, sufficient conditions are conservative by nature. That is, the existence of loads and network parameters that do not satisfy the conditions yet yield convergence is not precluded.  It is clear from conditions \eqref{eqn:selfmappingcondition} and \eqref{eqn:contractioncondition} that by increasing the constant-power and constant-current loads, there will eventually be no $R >0 $  that ensures contraction. 
	
As an example, in the IEEE-123 test feeder, after $50\%$ of load increase, no $R>0$ satisfies condition \eqref{eqn:selfmappingcondition}, however the Z-Bus converges to a solution within 5 iterations. Therefore, the ultimate goal would be to provide  a set of conditions that divide the voltage space to two regions, namely, one where there is convergence and another one of non-convergence. However, this is still an open research question. The next section provides two network examples where the Z-Bus method fails to converge in cases where conditions \eqref{eqn:linetoneutralcondition}--\eqref{eqn:contractioncondition} are not satisfied yet a solution may exist.

\section{Numerical examples of non-convergence}
\label{sec:non-convergence}
In this section, we provide two example networks where the Z-Bus method does not converge. The first example (Section \ref{sec:non-convergence:twonode}) depicts a simple two-node network for which an analytical solution is found; however, convergence conditions are not satisfied and the Z-Bus method oscillates.  The second example (Section \ref{sec:non-convergence:threenode}) features a non-trivial unbalanced network with constant-power loads where the convergence occurs only when the loads are reduced by a certain percentage. The convergence conditions also are satisfied for a close percentage, which reveals that they are relatively tight in this example. 
\subsection{Balanced two-node network with power injection}
\label{sec:non-convergence:twonode}
Consider a  two-node network where the slack bus $\mr{S}$ is connected  to a secondary bus through an ideal phase-decoupled three-phase line with real-valued series impedance $y_t$. The secondary bus contains real-valued ZIP components  $s_L$, $i_L$, and $y_L$ on each phase, where  $i_L, y_L >0$ but $s_L<0$, that is, the constant-power portion of the ZIP component is considered to be a power injection.  In this case, we have that $\mb{s}_L= s_L \mb{1}$, $\mb{i}_L=i_L \mb{1}$, $\mb{Y}_{\mr{L}}=y_L\mb{I}$ where $\mb{1} \in \mbb{R}^3$ is a vector of all one's and $\mb{I} \in \mbb{R}^{3\times 3}$ is the $3\times 3$ identity matrix.  Moreover, the network admittance matrix is $\mb{Y}=-\mb{Y}_{\mr{NS}}= y_t\mb{I}$.  Denote the voltage vector at the non-slack bus with $\mb{v}=\{v^a, v^b, v^c\}$ and let the slack bus voltage be fixed at $\mb{v}_{\mr{S}}=\{1, e^{-j\frac{2\pi}{3}}, e^{j\frac{2\pi}{3}}\}$. 

Due to the fact that the network is  balanced and the values of loads are real,  it turns out that a solution of the form $v^a=v^b e^{j\frac{2\pi}{3}}= v^c e^{-j\frac{2\pi}{3}}=v$ exists where $v$ is a real number. To demonstrate this, we write the Ohm's law for phase $a$ as follows:
\begin{IEEEeqnarray}{rCl}
	-\frac{s_L}{v}- i_L \frac{v}{|v|}= (y_t+y_L) v - y_t.  \label{eqn:twonodephaseavoltage}
	\end{IEEEeqnarray}

Assuming that we are only interested in non-negative voltages $v$, that is, when $|v|=v$, the following quadratic equation can be obtained for $v$: 
\begin{IEEEeqnarray}{rCl}
	(y_t+y_L) v^2 +(i_L-y_t)v +s_L=0.  \label{eqn:twonodequadratic}
\end{IEEEeqnarray}
The solution to \eqref{eqn:twonodequadratic} is explicitly given as 
\begin{IEEEeqnarray}{rCl}
	v= \frac{y_t - i_L}{2(y_t+y_L)} \pm  \frac{ \sqrt{\Delta}}{2 (y_t+y_L)}, \label{eqn:twonodegeneralsolution}
\end{IEEEeqnarray}
with $\Delta= (y_t-i_L)^2 - 4(y_t+y_L)s_L$.  Keep in mind that \eqref{eqn:twonodegeneralsolution} is not necessarily the only admissible voltage solution for this network since we have only limited the search to $v \in \mbb{R}_{\ge 0}$. 

In order to simplify the ensuing computations, let us select $y_t=y_L=\frac{1}{2}~\mr{pu}$, and set the current injection $i_L=\frac{1}{2}~\mr{pu}$. Substituting these values into \eqref{eqn:twonodegeneralsolution}, the real non-negative voltage solution $v$ parameterized  by $s_L$ is thus computed 
\begin{IEEEeqnarray}{rCl}
	v=  \sqrt{-s_L}. \label{eqn:twonodevsolution}
\end{IEEEeqnarray}
It is clear that for values $s_L < 0$, the network admits at least one  voltage solution $\mb{v}^{\mr{sol}}=\sqrt{-s_L}\mb{v}_{\mr{S}}$. 

At this point, it is our intention to show that even though   voltage solutions exist in this example network,  the Z-Bus method may not converge and the convergence conditions \eqref{eqn:linetoneutralcondition}--\eqref{eqn:contractioncondition} are not satisfied. 

We first compute $\mb{Z}= (\mb{Y}+\mb{Y}_{\mr{L}})^{-1} = \frac{1}{y_t+y_L} \mb{I}= \mb{I}$ and $\mb{w}=\frac{1}{2}\mb{v}_{\mr{S}}$. By selecting $\mb{\Lambda}=\mb{I}$, convergence conditions \eqref{eqn:linetoneutralcondition}--\eqref{eqn:contractioncondition} for this network yield the following inequalities parameterized by $s_L$:
\begin{subequations}
\label{eqngroup:twonodeconditions}
\begin{IEEEeqnarray}{rCl}
1 &-& 2R >0,  \label{eqn:twonodelinetoneutral} \\
\frac{2|s_L|}{1-2R} &+& \frac{1}{2} \le R, \label{eqn:twonodeselfmapping}\\
\frac{4 |s_L|}{(1-2R)^2} &+& \frac{2}{1-2R} <1, \label{eqn:twonodecontraction}
\end{IEEEeqnarray}
\end{subequations}
where in \eqref{eqngroup:twonodeconditions},   due to the absence of delta-connected nodes, \eqref{eqn:linetolinecondition} is not present.

Inequality \eqref{eqn:twonodeselfmapping} can be transformed to the quadratic inequality $2R^2-2R+\frac{1}{2}+2|s_L| \le 0$, which does not have a feasible solution in $R$ for any $s_L<0$. However,  existence of power solutions for $s_L<0$ is easily perceived by \eqref{eqn:twonodevsolution}.  Figure \ref{fig:twoNodeFigure} plots the infinity norm of the iterates of the Z-Bus method, that is, $\|\mb{v}[t]\|_{\infty}$,  when $s_L=-0.5~\mr{pu}$. Figure \ref{fig:twoNodeFigure} is representative of the fact that the Z-Bus method does not converge to the solution \eqref{eqn:twonodevsolution}.  
\begin{figure}[t]
	\centering
	\includegraphics[scale=0.45]{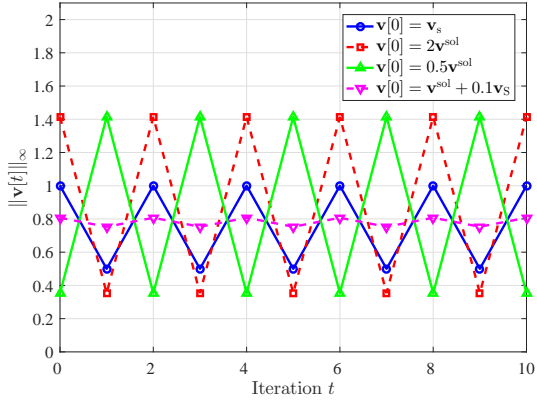} 	
	\caption{ \small Iterations of the Z-Bus method performed on the two-node test network while setting $s_L=-0.5~\mr{pu}$.  Since Z-Bus convergence conditions in \eqref{eqngroup:twonodeconditions} are not feasible for any $s_L$, the convergence of the Z-Bus method is not guaranteed, even when there is knowledge that  solution \eqref{eqn:twonodevsolution} exists.} 
		\label{fig:twoNodeFigure}
\end{figure}

\subsection{Unbalanced three-node network}
\label{sec:non-convergence:threenode}
Consider a three-node network with the set of nodes $\mc{N}=\{1,2,\mr{S}\}$ and the set of edges $\{(1,\mr{S}), (1,2)\}$. The series admittances for edges $(1,\mr{S})$ and $(1,2)$  are as follows: 
\begin{subequations}
	\label{eqngroup:threenodeseriesadmittance}
\begin{IEEEeqnarray}{rCl}
	\mb{Y}_{1\mr{S}}= \bmat{0.077 - j5.33 &  0.01 - j0.09  & 0.02 - j0.08\\
		0.01 - j0.09 &  0.087 - j8  & 0.03 - j0.07\\
		0.02 - j0.08  & 0.03 - j0.07 &  0.07 - j1.5}, \IEEEeqnarraynumspace \\
	\mb{Y}_{12}= \bmat{  0.056 - j8.66 &  0 & 0.02 - j0.07\\
	0 &   0.02 - j4.8 &  0.03 - j0.05\\
	0.02- j0.07 &  0.03 - j0.05  & 0.03 - j3.8}, \IEEEeqnarraynumspace
\end{IEEEeqnarray}
\end{subequations}
where all values are per unit. 
The load buses $\{1,2\}$ only contain constant-power loads with per unit values of  $s_{L_1}=\theta \bmat{0.7+j1.5 & 0.8+j1.5& 0.8+j2.5}^T$ and $s_{L_2}=\theta \bmat{0.6+j2.5 & 0.6+j0.5 & 0.3+j0.5}^T$ where $\theta \in (0,1]$ is a parameter that is going to be varied to obtain convergence threshold for the Z-Bus method.  

Our intention is to find $\theta$ such that the convergence of the Z-Bus method is guaranteed and numerically investigate how conservative the theoretical bound is.  First,  network parameters are computed: 
\begin{subequations}
	\label{eqngroup:threenodenetworkParams}
\begin{IEEEeqnarray}{rCl}
\mb{Y}&=&\bmat{\mb{Y}_{1\mr{S}}+\mb{Y}_{12} & -\mb{Y}_{12} \\ -\mb{Y}_{12} & \mb{Y}_{12} } \label{eqn:threenodeY} \\
\mb{Y}_{\mr{NS}}&=& \bmat{-\mb{Y}_{1\mr{S}} \\ \mb{O} }, \quad
\mb{w}= \bmat{ \mb{v}_{\mr{S}} \\ \mb{v}_{\mr{S}}}. \label{eqn:threenodeYNSw}
\end{IEEEeqnarray}
\end{subequations}

Next, to find $\theta$ such that Z-Bus convergence is guaranteed, by setting $\mb{\Lambda}=\mb{I}$, the convergence conditions \eqref{eqn:linetoneutralcondition}--\eqref{eqn:contractioncondition} yield the following inequalities in $R$ parameterized by $\theta$:
\begin{subequations}
\label{eqngroup:threenodeconditions}
\begin{IEEEeqnarray}{rCl}
1-R>0 &\Rightarrow& R<1 \label{eqn:threenodec1} \\
\frac{2.359 \theta}{1-R} \le R &\Rightarrow&  R^2-R+2.359 \theta \le 0 \label{eqn:threenodec3}  \\
\frac{2.359 \theta}{(1-R)^2} < 1 &\Rightarrow& R^2- 2R+1-2.359\theta >0. \label{eqn:threenodec4}
\end{IEEEeqnarray}
\end{subequations}
It is not hard to see that \eqref{eqngroup:threenodeconditions}  has a valid solution for $R$ only if $\theta \le \theta^{\mr{sol}}= 0.1060$. The range of $R$ for convergent Z-Bus is then given as $[R_{\min}, R_{\max}]=[\frac{1-\sqrt{1- 9.4360 \theta}}{2},\frac{1+\sqrt{1- 9.4360 \theta}}{2} ]$.  
The Z-Bus method is run on this example network for several representative values of $\theta$ using the initialization $\mb{v}[0]=\mb{w}$.  The corresponding Z-Bus iterations are shown  in Fig.~\ref{fig:threeNodeFigure}. It is revealed in Fig.~\ref{fig:threeNodeFigure}, that for almost all percentages $\theta$ higher than $\theta^{\mr{sol}}$ the Z-Bus method fails to converge. However, for $\theta \le \theta^{\mr{sol}}$, the Z-Bus iterations are guaranteed to converge to the unique solution in $\mb{D}_R$ where $R \in [\frac{1-\sqrt{1- 9.4360 \theta}}{2},\frac{1+\sqrt{1- 9.4360 \theta}}{2} ]  $.

\begin{figure}[t]
	\centering
	\includegraphics[scale=0.45]{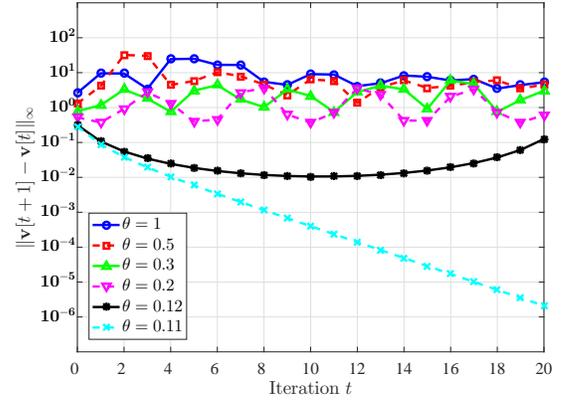} 	
	\caption{ \small Iterations of the Z-Bus method performed on the three-node test network for several values of parameter $\theta$.   When the value of $\theta$ is such that Z-Bus convergence conditions \eqref{eqngroup:threenodeconditions} are not feasible, the convergence of the Z-Bus is not guaranteed. The black line corresponding to $\theta=0.12$ depicts a very long oscillation.  As $\theta$ approaches $\theta^{\mr{sol}}=0.1060$, the Z-Bus method converges. When $\theta \le \theta^{\mr{sol}}$, a unique solution exists within $\mb{D}_R$ where $R \in [\frac{1-\sqrt{1- 9.4360 \theta}}{2},\frac{1+\sqrt{1- 9.4360 \theta}}{2} ]  $.
		\label{fig:threeNodeFigure}}
\end{figure}

\section{Concluding Remarks}
\label{sec:conclusion}
A set of sufficient conditions was derived that guarantee the convergence of the Z-Bus method to a unique solution of the power flow problem in three-phase distribution networks with ZIP loads. It was numerically demonstrated that the IEEE test feeders with wye and delta ZIP loads satisfy the derived conditions. Characterizing the largest regions with unique power flow solution is an open research question. Investigating the effect of the design parameter $\mb{\Lambda}$ to this end is the subject of future work.

\appendices
\section{Useful Intermediate Results}
\label{sec:app:voltagebounds}
This Appendix proves intermediate results on the product $\mb{Z}[\mb{i}_{\mr{PQ}} + \mb{i}_{\mr{I}}]$ that shows up in \eqref{eqn:zbusv}, and an inequality on complex numbers, which will be later used. 

\subsection{Rearranging $\sum\nolimits_{k \in \mc{J}_\Delta} \mb{Z}_{jk} \mb{i}_{\mr{PQ}}(\mb{v})_k$ } 
\label{subsec:lemma1}
\begin{lemma}
\label{lemma1}
Consider the matrix $\mb{Z}^{\Delta}$ defined in \eqref{eqn:zdelta}. 
For  $n \in \mc{N}_{\Delta}$, $\phi \in \Omega_n$,  $k =\texttt{Lin}(n,\phi) \in \mc{J}_{\Delta}$, $j \in \mc{J}$,  and indices $l_k \in \{1, \ldots, |\mc{J}_\Delta|\}$, the following equality holds:
\begin{IEEEeqnarray}{rCl}
\sum\limits_{k \in \mc{J}_\Delta} \mb{Z}_{jk} \mb{i}_{\mr{PQ}}(\mb{v})_k = -\sum\limits_{k \in \mc{J}_\Delta} \mb{Z}_{jl_k}^{\Delta} \left(\frac{s_L^k}{\mb{e}_k^T \mb{v}_n}\right)^*.    \label{eqn:lemma1}
\end{IEEEeqnarray}
\end{lemma}
\begin{proof}
Using \eqref{eqn:deltaloadPQ} for constant-PQ loads with delta connections, we can write
\begin{IEEEeqnarray}{rCl}
\IEEEeqnarraymulticol{3}{l}{\sum\limits_{k \in \mc{J}_\Delta} \mb{Z}_{jk} \mb{i}_{\mr{PQ}}(\mb{v})_k  =} \notag \\
&& -\sum\limits_{n \in \mc{N}_\Delta} \left(\sum\limits_{\phi \in \Omega_n} \mb{Z}_{j\texttt{Lin}(n,\phi)} \sum\limits_{\phi' \in \Omega_n / \{\phi\}} \frac{s_{L_n}^{\phi \phi^{'*}}}{(\mb{e}_n^{\phi \phi'})^T \mb{v}_n^*}\right). \label{eqn:ideltaPQfirstParanthesis} \IEEEeqnarraynumspace
\end{IEEEeqnarray}
Consider the term in parenthesis in \eqref{eqn:ideltaPQfirstParanthesis}.  Since $n \in \mc{N}_\Delta$, then $|\Omega_n| \ge 2$, i.e., node $n$ has at least two available phases.  
If $|\Omega_n|=2$, then without loss of generality we can assume $\Omega_n=\{\phi, \phi'\}$ where $\phi'=r(\phi)$ and $r(\phi') \notin \Omega_n$.  The term in parenthesis in \eqref{eqn:ideltaPQfirstParanthesis} can be written as follows:

\begin{IEEEeqnarray}{rCl}
\IEEEeqnarraymulticol{3}{l}{\mb{Z}_{j \texttt{Lin}(n,\phi)} \frac{ s_{L_n}^{\phi \phi^{'*}}}{ (\mb{e}_n^{\phi\phi'})^T \mb{v}_n^*} + \mb{Z}_{j \texttt{Lin}(n,\phi')} \frac{ s_{L_n}^{\phi' \phi^*}}{ (\mb{e}_n^{\phi'\phi})^T \mb{v}_n^*} } \notag\\
&=&\mb{Z}_{j\texttt{Lin}(n,\phi)}\frac{ s_{L_n}^{\phi r(\phi)^*}}{ (\mb{e}_n^{\phi r(\phi)})^T \mb{v}_n^*} - \mb{Z}_{j \texttt{Lin}(n,r(\phi))} \frac{ s_{L_n}^ {\phi r(\phi)^*}}{ (\mb{e}_n^{\phi r(\phi)})^T \mb{v}_n^*}  \notag \\
&=&\left[\mb{Z}_{jk} - \mb{Z}_{j k'}\right]\frac{ s_L^{k^*}}{ \mb{e}_k^T \mb{v}_n^*} = \mb{Z}_{jl_k}^\Delta \frac{ s_L^{k^*}}{ \mb{e}_k^T \mb{v}_n^{*}} \stackrel{\text{(a)}}= \sum\limits_{k \in \mc{J}_n} \mb{Z}_{jl_k}^{\Delta} \frac{ s_L^{k^*}}{ \mb{e}_k^T \mb{v}_n^*},  \IEEEeqnarraynumspace \label{eqn:ideltaPQParanthesisTwoPhases}
\end{IEEEeqnarray}
where $k'=\texttt{Lin}(n,r(\phi))$, and equality (a) results from the fact that $\mb{Z}_{jl_k}^{\Delta}$ is zero when $\phi \in \Omega_n$ and $r(\phi) \notin \Omega_n$ [cf. \eqref{eqn:zdelta}]. 

If node $n$ has three phases then $\Omega_n=\{a,b,c\}$ and $r(a)=b$, $r(b)=c$, and $r(c)=a$.  The term in parenthesis in \eqref{eqn:ideltaPQfirstParanthesis} equals:
\begin{IEEEeqnarray}{rCl}
&& \mb{Z}_{j \texttt{Lin}(n,a)} \frac{ s_{L_n}^{ab^*}}{ (\mb{e}_n^{ab})^T \mb{v}_n^*} + \mb{Z}_{j \texttt{Lin}(n,a)} \frac{ s_{L_n}^{ac^*}}{ (\mb{e}_n^{ac})^T \mb{v}_n^*} \notag \\
&& + \:  \mb{Z}_{j \texttt{Lin}(n,b)} \frac{ s_{L_n}^{ba^*}}{ (\mb{e}_n^{ba})^T \mb{v}_n^*} + \mb{Z}_{j \texttt{Lin}(n,b)} \frac{ s_{L_n}^{bc^*}}{ (\mb{e}_n^{bc})^T \mb{v}_n^*} \notag \\
 && + \:\mb{Z}_{j \texttt{Lin}(n,c)} \frac{s_{L_n}^{ca^*}} { (\mb{e}_n^{ca})^T \mb{v}_n^*} +  \mb{Z}_{j \texttt{Lin}(n,c)} \frac{s_{L_n}^{cb^*}} { (\mb{e}_n^{cb})^T \mb{v}_n^*} \notag \\
 &=&\mb{Z}_{j \texttt{Lin}(n,a)} \frac{ s_{L_n}^{ab^*}}{ (\mb{e}_n^{ab})^T \mb{v}_n^*} - \:  \mb{Z}_{j \texttt{Lin}(n,b)} \frac{ s_{L_n}^{ab^*}}{ (\mb{e}_n^{ba})^T \mb{v}_n^*}  \notag \\
 &&  + \:  \mb{Z}_{j \texttt{Lin}(n,b)} \frac{ s_{L_n}^{b c^*}}{ (\mb{e}_n^{bc})^T \mb{v}_n^*} - \mb{Z}_{j \texttt{Lin}(n,c)} \frac{s_{L_n}^{bc^*}} { (\mb{e}_n^{bc})^T \mb{v}_n^*}  \notag \\
 && + \: \mb{Z}_{j \texttt{Lin}(n,c)} \frac{s_{L_n}^{ca^*}} { (\mb{e}_n^{ca})^T \mb{v}_n^*}  - \mb{Z}_{j \texttt{Lin}(n,a)} \frac{ s_{L_n}^{ca^*}}{ (\mb{e}_n^{ca})^T \mb{v}_n^*} \notag  \\
&=& \sum\limits_{\phi \in \Omega_n} \left[\mb{Z}_{j \texttt{Lin}(n,\phi)}- \mb{Z}_{j \texttt{Lin}(n,r(\phi))}\right] \frac{s_{L_n}^{\phi r(\phi)^*}} {(\mb{e}_n^{\phi r(\phi)})^T \mb{v}_n^*} \notag \\
&=& \sum\limits_{k \in \mc{J}_n} \left[\mb{Z}_{jk} - \mb{Z}_{jk'}\right]  \frac{s_L^{k^*}}{\mb{e}_k^T \mb{v}_n^*} = \sum\limits_{k \in \mc{J}_n} \mb{Z}_{jl_k}^{\Delta} \frac{s_L^{k^*}}{\mb{e}_k^T \mb{v}_n^*}. \IEEEeqnarraynumspace \label{eqn:ideltaPQParanthesisThreePhase} 
\end{IEEEeqnarray}
For $|\Omega_n|=2$ or $|\Omega_n|=3$, substituting the appropriate term from \eqref{eqn:ideltaPQParanthesisTwoPhases} or \eqref{eqn:ideltaPQParanthesisThreePhase} for the parenthesis in \eqref{eqn:ideltaPQfirstParanthesis} yields \eqref{eqn:lemma1}. 
\end{proof}
\subsection{Rearranging $\sum\nolimits_{k \in \mc{J}_{\Delta}} \mb{Z}_{jk} \mb{i}_{\mr{I}} (\mb{v})_k$}

\begin{lemma}
\label{lemma2}
Consider the matrix $\mb{Z}^{\Delta}$ defined in \eqref{eqn:zdelta}. 
For  $n \in \mc{N}_{\Delta}$, $\phi \in \Omega_n$,  $k =\texttt{Lin}(n,\phi) \in \mc{J}_{\Delta}$, $j \in \mc{J}$,  and indices $l_k \in \{1, \ldots, |\mc{J}_\Delta|\}$, the following equality holds:
\begin{IEEEeqnarray}{rCl}
\sum\limits_{k \in \mc{J}_\Delta} \mb{Z}_{jk} \mb{i}_{\mr{I}}(\mb{v})_k =  \sum\limits_{k \in \mc{J}_n} \mb{Z}_{jl_k}^{\Delta} \frac{i_L^k \mb{e}_k^T \mb{v}_n}{|\mb{e}_k^T \mb{v}_n|}.    \label{eqn:lemma2}
\end{IEEEeqnarray}
\end{lemma}
\begin{proof}
Follow Appendix \ref{subsec:lemma1} with the terms $\frac{s_{L_n}^{\phi \phi^{'*}}}{(\mb{e}_n^{\phi \phi'})^T \mb{v}_n^*}$ and $\frac{s_L^{k^*}}{\mb{e}_k^T \mb{v}_n^*} $ replaced by $\frac{i_{L_n}^{\phi \phi'} (\mb{e}_n^{\phi \phi'})^T \mb{v}_n}{|(\mb{e}_n^{\phi \phi'})^T \mb{v}_n|}$ and  $\frac{ i_L^k  \mb{e}_k^T \mb{v}_n}{ |\mb{e}_k^T \mb{v}_n|} $, respectively.
 \end{proof}

\subsection{A useful inequality for complex numbers}
\begin{lemma}
\label{lemma3}
For two complex numbers $x,y \neq 0$ we have that
\begin{IEEEeqnarray}{rCl}
\left|\frac{x}{|x|} - \frac{y}{|y|}\right| \le  2 \frac{ |x -y|}{|x|}.
\end{IEEEeqnarray}
\end{lemma}
\begin{proof} The following hold:
\begin{IEEEeqnarray}{rCl} 
\left|\frac{x}{|x|} - \frac{y}{|y|}\right| &\le & \left|\frac{x}{|x|} - \frac{y}{|x|}\right| + \left|\frac{y}{|x|} - \frac{y}{|y|}\right| \notag \\
&=& \frac{ |x -y|}{|x|} + |y| \frac{\big| |y|- |x| \big|}{|x||y|} \le 2 \frac{ |x -y|}{|x|}.  \label{eqn:lemma3proof}
\end{IEEEeqnarray}
Notice that \eqref{eqn:lemma3proof} holds since $| |y|- |x|| \le |x-y|$.
\end{proof}
\section{Proof of Theorem \ref{thm:main}}
\label{sec:appendixProof}
In this section, the proof of Theorem \ref{thm:main} is given.  First, to guarantee that for all $\mb{u} \in \mb{D}_R$  the corresponding line-to-neutral and line-to-line voltage magnitudes are positive,  conditions  \eqref{eqn:linetoneutralcondition} and \eqref{eqn:linetolinecondition} are derived in Appendices \ref{subsec:linetoneutralbounds} and \ref{subsec:linetolinebounds} respectively.   Condition \eqref{eqn:selfmappingcondition} ensures the self-mapping property of $\mb{T}(\mb{u})$ and is derived in Appendix \ref{subsec:selfmapping}.   Appendix  \ref{subsec:contraction} derives condition \eqref{eqn:contractioncondition},  which guarantees the contraction property of $\mb{T}(\mb{u})$.
\subsection{ Bounds on line to neutral voltages}
\label{subsec:linetoneutralbounds}
This section proves two results. First, a lower bound on all line to neutral voltages is derived in \eqref{eqn:linetoneutralvoltagesLowerBound} which will be later  used in proving the self-mapping and contraction properties.  Second,  condition \eqref{eqn:linetoneutralcondition} is proved.   

Due to condition $\mb{u} \in \mb{D}_R$, the following holds for $k \in \mc{J}$:
\begin{IEEEeqnarray}{rCl}
&& \max_{k \in \mc{J}} | u_k- \frac{w_k}{\lambda_k} | \le  R  \Rightarrow |u_k- \frac{w_k}{\lambda_k}| \le R,   \notag  \\
&& \Rightarrow | \lambda_ku_k - w_k| \le |\lambda_k| R \Rightarrow |v_k - w_k| \le R|\lambda_k|, \notag \\
&& \Rightarrow |w_k| - |v_k| \le R |\lambda_k|  \Rightarrow  |w_k| - R|\lambda_k| \le |v_k| \label{eqn:linetoneutralvoltagesLowerBound1} \\
&&\Rightarrow |w_k|( 1- R|\lambda_k|/|w_k|) \le |v_k|. \label{eqn:linetoneutralvoltagesLowerBound}  \IEEEeqnarraynumspace
\end{IEEEeqnarray} 
Equation \eqref{eqn:linetoneutralvoltagesLowerBound} provides a lower bound for the magnitude of line to neutral voltages in terms of $R$ and $\mb{w}$.   

The proof of \eqref{eqn:linetoneutralcondition} is based on \eqref{eqn:linetoneutralvoltagesLowerBound}. 	  In \eqref{eqn:linetoneutralvoltagesLowerBound},  in the term in parenthesis, we replace $|w_k|$ by its minimum $\underline{w}$ [defined in \eqref{eqnGroup:knownminimums}], and replace $|\lambda_k|$ by its maximum $\bar{\lambda}$ [also defined in \eqref{eqnGroup:knownminimums}]  to obtain a lower bound for the left-hand side of \eqref{eqn:linetoneutralvoltagesLowerBound} which is common for all $k \in \mc{J}$:
\begin{IEEEeqnarray}{rCl}
|w_k|(1- R |\lambda_k|/|w_k|) \ge |w_k| ( 1 - R\bar{\lambda}/\underline{w}). \IEEEeqnarraynumspace \label{eqn:linetoneutralvoltagesLowestBound} 
\end{IEEEeqnarray}
To ensure that for all $\mb{u} \in \mb{D}_R$,  the magnitude of the line to neutral voltages is positive, the right-hand side of \eqref{eqn:linetoneutralvoltagesLowestBound} must be greater than zero, yielding condition \eqref{eqn:linetoneutralcondition}. 
\subsection{Bounds on line to line voltages} 
\label{subsec:linetolinebounds}
Similar to Appendix \ref{subsec:linetoneutralbounds}, this section proves two results. First,  we find a lower bound for line to line voltages in delta connections.  Next we derive condition \eqref{eqn:linetolinecondition}. 
  
For $n \in \mc{N}_{\Delta}$, any line to line voltage can be written as $\mb{e}_k^T \mb{v}_n$, where $k=\texttt{Lin}(n,\phi)$ and $\phi \in \Omega_n$.  Using the triangle inequality, we find the following:
\begin{IEEEeqnarray}{rCl}
 |\mb{e}_k^T \mb{v}_n| &=& |\mb{e}_k^T(\mb{v}_n-\mb{w}_n)+ \mb{e}_k^T \mb{w}_n|  \notag \\
&\ge& |\mb{e}_k^T \mb{w}_n| - |\mb{e}_k^T(\mb{v}_n-\mb{w}_n)| \label{eqn:evw}
\end{IEEEeqnarray}
Due to H\"older's inequality, for $n=\texttt{Node}[k]$ we have that 
\begin{IEEEeqnarray}{rCl}
|\mb{e}_k^T(\mb{v}_n-\mb{w}_n)| &\le& \|\mb{e}_k\|_1 \| \mb{v}_n- \mb{w}_n\|_{\infty},  \notag  \\
\Rightarrow |\mb{e}_k^T(\mb{v}_n-\mb{w}_n)|  & \stackrel{\text{(a)}}\le&  2 \max_{l \in \mc{J}_n} \{ |v_l - w_l|\} \notag \\
\Rightarrow  |\mb{e}_k^T(\mb{v}_n-\mb{w}_n)| &\le& 2 \max_{l \in \mc{J}_n} \{|\lambda_l| |u_l - \frac{w_l}{\lambda_l}| \}  \notag \\
\Rightarrow  |\mb{e}_k^T(\mb{v}_n-\mb{w}_n)| &\stackrel{\text{(b)}}\le& 2R \max_{l \in \mc{J}_n} \{|\lambda_l| \},
 \label{eqn:holdersuse}
\end{IEEEeqnarray}
where (a) holds because $\|\mb{e}_k\|_{1}=2$ for $k \in \mc{J}$,  and (b) holds due to the assumption $\mb{u} \in \mb{D}_R$.  Using \eqref{eqn:holdersuse} in \eqref{eqn:evw} we find:
\begin{IEEEeqnarray}{rCl}
 |\mb{e}_k^T \mb{v}_n| &\ge& |\mb{e}_k^T\mb{w}_n| - 2R \max_{l \in \mc{J}_n} \{|\lambda_l| \} \label{eqn:linetolinevoltagesLowerBound1} \\
& \ge &|\mb{e}_k^T\mb{w}_n| (1 - 2 R \max_{l \in \mc{J}_n} \{ |\lambda_l| \} / |\mb{e}_k^T\mb{w}_n|). \label{eqn:linetolinevoltagesLowerBound}
\end{IEEEeqnarray}
 Equation \eqref{eqn:linetolinevoltagesLowerBound} provides a lower bound for the magnitude of line to line voltages in terms of $R$ and $\mb{w}$. 

The proof of \eqref{eqn:linetolinecondition} is based on \eqref{eqn:linetolinevoltagesLowerBound}.  In \eqref{eqn:linetolinevoltagesLowerBound},  in the term in parenthesis, we replace the term $|\mb{e}_k^T\mb{w}_n|$ by its minimum $\underline{\rho}$ [given  in \eqref{eqnGroup:knownminimums}], and replace $|\lambda_k|$ by its maximum $\bar{\lambda}$ [also given in \eqref{eqnGroup:knownminimums}] to obtain a lower bound for the line to line voltages  which is common for all $k \in \mc{J}_{\Delta}$ and $n=\texttt{Node}[k]$:
\begin{IEEEeqnarray}{rCl}
\IEEEeqnarraymulticol{3}{l}{ |\mb{e}_k^T\mb{w}_n| (1 - 2 R \max_{l \in \mc{J}_n} \{ |\lambda_l| \} / |\mb{e}_k^T\mb{w}_n|)}  \notag \\
 && \ge  |\mb{e}_k^T\mb{w}_n| (1 - 2R \bar{\lambda} / \underline{\rho}).  \label{eqn:linetolinevoltagesLowestBound}\IEEEeqnarraynumspace
\end{IEEEeqnarray} 
To guarantee that for all $\mb{u} \in \mb{D}_R$, the magnitude of line-to-line voltages is positive,  \eqref{eqn:linetolinevoltagesLowestBound} must be greater than zero, yielding condition \eqref{eqn:linetolinecondition}.

\subsection{Self-mapping property}
\label{subsec:selfmapping}
 To obtain a sufficient condition for the self-mapping property, it is assumed that $\mb{u} \in \mb{D}_R$, and an upper bound for the term $\|\mb{T}(\mb{u}) - \mb{\Lambda}^{-1} \mb{w} \|_{\infty}$ is sought.  The $j$-th entry of  $\|\mb{T}(\mb{u}) - \mb{\Lambda}^{-1} \mb{w} \|_{\infty} $ can be upper-bounded as follows:
\begin{IEEEeqnarray}{rCl}
| \mb{T}(\mb{u})_j - \frac{w_j}{\lambda_j}| &=& \left|\frac{1}{\lambda_j} \mb{Z}_{j\bullet} \left[ \mb{i}_{\mr{PQ}}(\mb{v})+ \mb{i}_{\mr{I}}(\mb{v})\right] + \frac{w_j}{\lambda_j}-\frac{w_j}{\lambda_j}\right| \notag \\
\IEEEeqnarraymulticol{3}{l}{\le \frac{1}{|\lambda_j|} \left(\left|\sum\limits_{k \in \mc{J}_Y} \mb{Z}_{jk} \mb{i}_{\mr{PQ}}(\mb{v})_k\right| + \left| \sum\limits_{k \in \mc{J}_\Delta } \mb{Z}_{jk} \mb{i}_{\mr{PQ}}(\mb{v})_k \right| \right.} \notag \\
\IEEEeqnarraymulticol{3}{l}{\: \: + \: \left.  \left| \sum\limits_{k \in \mc{J}_Y} \mb{Z}_{jk} \mb{i}_\mr{I} (\mb{v}) \right| + \left| \sum\limits_{k \in \mc{J}_\Delta} \mb{Z}_{jk} \mb{i}_{\mr{I}} (\mb{v}) \right| \right).}
 \label{eqn:transformball1} \IEEEeqnarraynumspace  
\end{IEEEeqnarray}

The right-hand side of \eqref{eqn:transformball1} is the sum of four terms.  An upper bound for each term is provided next. 

For the first term in \eqref{eqn:transformball1}, the following holds: 
\begin{IEEEeqnarray}{rCl}
\left| \sum\limits_{k \in \mc{J}_Y} \mb{Z}_{jk} \mb{i}_{\mr{PQ}} (\mb{v})_k \right| &=& \left| \sum \limits_{ k \in \mc{J}_Y} \mb{Z}_{jk} ( \frac{s_{L}^k}{v_k})^* \right| \notag \IEEEeqnarraynumspace \\
\IEEEeqnarraymulticol{3}{l}{= \left| \sum\limits_{k \in \mc{J}_Y} \mb{Z}_{jk} \left[ \left(\frac{s_L^k}{v_k}\right)^* - \left(\frac{s_L^k}{w_k}\right)^*+  \left(\frac{s_L^k}{w_k}\right)^* \right] \right|} \notag \\
\IEEEeqnarraymulticol{3}{l} { \le  \sum\limits_{k \in \mc{J}_Y} |\mb{Z}_{jk}| \frac{ |s_L^k| |w_k-v_k|}{|v_k| |w_k|}  + \sum \limits_{k \in \mc{J}_Y} |\mb{Z}_{jk}|  \frac{ |s_{L}^k|}{|w_k|}} \notag \\
\IEEEeqnarraymulticol{3}{l} { \stackrel{\text{(a)}}{\le}  \sum\limits_{k \in \mc{J}_Y} |\mb{Z}_{jk}| \frac{ |s_L^k| R |\lambda_k|}{\left(|w_k| - R |\lambda_k|\right) |w_k|}  + \sum \limits_{k \in \mc{J}_Y} |\mb{Z}_{jk}|  \frac{ |s_{L}^k|}{|w_k|}} \notag \\
\IEEEeqnarraymulticol{3}{l} { =  \sum\limits_{k \in \mc{J}_Y} \frac{|\mb{Z}_{jk}||s_L^k|}{|w_k| - R |\lambda_k|}  \stackrel{\text{(b)}}{\le} \sum\limits_{k \in \mc{J}_Y} \frac{|\mb{Z}_{jk}||s_L^k|}{|w_k| (1- R \bar{\lambda} / \underline{w})},  }  \label{eqn:sYUpperBound}
\end{IEEEeqnarray}
where  (a) is due to \eqref{eqn:linetoneutralvoltagesLowerBound1}, and  (b) is due to \eqref{eqn:linetoneutralvoltagesLowerBound} and \eqref{eqn:linetoneutralvoltagesLowestBound}. 

For the second term in \eqref{eqn:transformball1}, we use Lemma \ref{lemma1}, and  we write
\begin{IEEEeqnarray}{rCl}
\left| \sum\limits_{k \in \mc{J}_\Delta} \mb{Z}_{jk} \mb{i}_{\mr{PQ}} (\mb{v})_k \right| &=& \left| \sum \limits_{ k \in \mc{J}_\Delta} \mb{Z}_{jl_k}^{\Delta} \left( \frac{s_{L}^k}{\mb{e}_k^T\mb{v}_n}\right)^* \right| \notag \IEEEeqnarraynumspace \\
\IEEEeqnarraymulticol{3}{l}{= \left| \sum\limits_{k \in \mc{J}_\Delta} \mb{Z}_{jl_k}^{\Delta} \left[ \left(\frac{s_L^k}{\mb{e}_k^T\mb{v}_n}\right)^* - \left(\frac{s_L^k}{\mb{e}_k^T\mb{w}_n}\right)^*+  \left(\frac{s_L^k}{\mb{e}_k^T\mb{v}_n}\right)^* \right] \right|} \notag \\
\IEEEeqnarraymulticol{3}{l} { \le  \sum\limits_{k \in \mc{J}_\Delta} |\mb{Z}_{jl_k}^{\Delta}| \frac{ |s_L^k| |\mb{e}_k^T(\mb{w}_n-\mb{v}_n)|}{|\mb{e}_k^T\mb{v}_n||\mb{e}_k^T\mb{w}_n|}  +  \sum \limits_{k \in \mc{J}_\Delta} |\mb{Z}_{jl_k}^{\Delta}|  \left|\frac{ s_{L}^k}{\mb{e}_k^T\mb{w}_n}\right|} \notag \\
\IEEEeqnarraymulticol{3}{l} { \stackrel{\text{(a)}}{\le}  \sum\limits_{k \in \mc{J}_\Delta} |\mb{Z}_{jl_k}^{\Delta}| \frac{ |s_L^k| 2R  \max\limits_{l \in \mc{J}_n} \{ |\lambda_l|\}}{(|\mb{e}_k^T\mb{w}_n| - 2R  \max\limits_{l \in \mc{J}_n} \{ |\lambda_l|\}) |\mb{e}_k^T\mb{w}_n|} }  \notag \\
\quad + \sum \limits_{k \in \mc{J}_\Delta} |\mb{Z}_{jl_k}^{\Delta}|  \left|\frac{ s_{L}^k}{\mb{e}_k^T\mb{w}_n}\right|\notag \\
\IEEEeqnarraymulticol{3}{l} { \le  \sum\limits_{k \in \mc{J}_\Delta} \frac{|\mb{Z}_{jl_k}^{\Delta}| |s_L^k|}{|\mb{e}_k^T\mb{w}_n|} \left[ \frac{2R  \max_{l \in \mc{J}_n} \{ |\lambda_l|\}}{|\mb{e}_k^T\mb{w}_n| - 2R  \max_{l \in \mc{J}_n} \{ |\lambda_l|\} } +1 \right] }  \notag \\
\IEEEeqnarraymulticol{3}{l} { \le  \sum\limits_{k \in \mc{J}_\Delta} \frac{|\mb{Z}_{jl_k}^{\Delta}| |s_L^k|}{|\mb{e}_k^T\mb{w}_n| (1  - 2R  \max_{l \in \mc{J}_n} \{ |\lambda_l|\}/|\mb{e}_k^T \mb{w}_n|)}}\notag \\
\IEEEeqnarraymulticol{3}{l} { \stackrel{\text{(b)}}{\le}  \sum\limits_{k \in \mc{J}_\Delta} \frac{|\mb{Z}_{jl_k}^{\Delta}| |s_L^k|}{|\mb{e}_k^T\mb{w}_n| (1  - 2R  \bar{\lambda}/\underline{\rho})},} \label{eqn:sDeltaUpperBound}
\end{IEEEeqnarray}
where $n=\texttt{Node}[k]$, index $l_k$ is defined in \eqref{eqn:zdelta},  inequality (a) is due to \eqref{eqn:holdersuse} and \eqref{eqn:linetolinevoltagesLowerBound1}, and inequality (b) is due to \eqref{eqn:linetolinevoltagesLowestBound}. 

The  upper bound for the third term in \eqref{eqn:transformball1}, is given by
\begin{IEEEeqnarray}{rCl}
\left| \sum\limits_{k \in \mc{J}_Y} \mb{i}_{\mr{I}}(\mb{v})\right| = \left| \sum\limits_{k \in \mc{J}_Y} \mb{Z}_{jk} \frac{v_k}{|v_k|} i_L^k \right| \le  \sum\limits_{k \in \mc{J}_Y} |\mb{Z}_{jk}| |i_L^k|. \IEEEeqnarraynumspace \label{eqn:iYUpperBound}
\end{IEEEeqnarray}

  Lemma \ref{lemma2} is invoked to provide an upper bound for  the fourth term in  \eqref{eqn:transformball1}:
\begin{IEEEeqnarray}{rCl}
\left|\sum\limits_{k \in \mc{J}_{\Delta}} \mb{Z}_{jk} \mb{i}_{\mr{I}} (\mb{v})_k \right| &=& \left|\sum\limits_{k \in \mc{J}_{\Delta}} \mb{Z}_{jl_k}^{\Delta} \frac{\mb{e}_k^T \mb{v}_n}{|\mb{e}_k^T \mb{v}_n|} i_{L}^k \right| \notag \\
\Rightarrow \left|\sum\limits_{k \in \mc{J}_{\Delta}} \mb{Z}_{jk} \mb{i}_{\mr{I}} (\mb{v})_k \right| &\le& \sum\limits_{k \in \mc{J}_{\Delta}} |\mb{Z}_{jl_k}^{\Delta}| |i_{L}^k|. \IEEEeqnarraynumspace \label{eqn:iDeltaUpperBound}
\end{IEEEeqnarray}
Using equations \eqref{eqn:sYUpperBound}, \eqref{eqn:sDeltaUpperBound}, \eqref{eqn:iYUpperBound}, and \eqref{eqn:iDeltaUpperBound} in \eqref{eqn:transformball1} yields
\begin{IEEEeqnarray}{rCl}
| \mb{T}(\mb{u})_j - \frac{w_j}{\lambda_j}|  &\le &  \frac{1}{|\lambda_j|} \left[  \sum\limits_{k \in \mc{J}_Y} \frac{|\mb{Z}_{jk}||s_L^k|}{|w_k| (1- R \bar{\lambda} / \underline{w})} \right.  \notag \\
  \IEEEeqnarraymulticol{3}{l}{\left.  + \:  \sum\limits_{k \in \mc{J}_\Delta} \frac{|\mb{Z}_{jl_k}^{\Delta}| |s_L^k|}{|\mb{e}_k^T\mb{w}_{\texttt{Node}[k]}| (1  - 2R  \bar{\lambda}/\underline{\rho})} \right.} \notag \\
 \IEEEeqnarraymulticol{3}{l}{\left. + \: \sum\limits_{k \in \mc{J}_Y} |\mb{Z}_{jk}| |i_L^k| + \sum\limits_{k \in \mc{J}_{\Delta}} |\mb{Z}_{jl_k}^{\Delta}| |i_{L}^k|\right]. }\label{eqn:upperboundforallcurrents}
\end{IEEEeqnarray}
The self-mapping property holds if the right-hand side of \eqref{eqn:upperboundforallcurrents} is upper-bounded by $R$ for all $j \in \mc{J}$, which is ensured by \eqref{eqn:selfmappingcondition}. 
\subsection{Contraction property}
\label{subsec:contraction}
To obtain a sufficient condition for the contraction property, we seek to find an upper-bound for the term $\|\mb{T}(\mb{u})-\mb{T}(\tilde{\mb{u}})\|_{\infty}$ proportional to  $\|\mb{u} - \tilde{\mb{u}}\|_{\infty}$.  Setting $\mb{v}=\mb{\Lambda}\mb{u}$ and $\tilde{\mb{v}} =\mb{\Lambda} \tilde{\mb{u}}$, the $j$-th entry of $\mb{T}(\mb{u})-\mb{T}(\tilde{\mb{u}})$ is upper-bounded as follows:
\begin{IEEEeqnarray}{rCl}
\IEEEeqnarraymulticol{3}{l}{|\mb{T}(\mb{u})_j - \mb{T}(\tilde{\mb{u}})_j| } \notag \\
\IEEEeqnarraymulticol{3}{l}{= \left|\frac{1}{\lambda_j} \mb{Z}_{j\bullet} \left[\mb{i}_{\mr{PQ}}(\mb{v})+ \mb{i}_{\mr{I}}(\mb{v})\right]  - \: \frac{1}{\lambda_j}\mb{Z}_{j\bullet}\left[ \mb{i}_{\mr{PQ}}(\tilde{\mb{v}})+ \mb{i}_{\mr{I}}(\tilde{\mb{v}}) \right]\right|} \notag \\
\IEEEeqnarraymulticol{3}{l}{\le \frac{1}{|\lambda_j|}  \left(   \left| \sum\limits_{k \in \mc{J}_Y}  \mb{Z}_{jk} \left[ \mb{i}_{\mr{PQ}}(\mb{v})_k - \mb{i}_{\mr{PQ}}(\tilde{\mb{v}})_k\right] \right| \right. } \IEEEeqnarraynumspace \notag \\
&&  \left. + \:  \left| \sum\limits_{k \in \mc{J}_\Delta} \mb{Z}_{jk} \left[ \mb{i}_{\mr{PQ}} (\mb{v})_k - \mb{i}_{\mr{PQ}}(\tilde{\mb{v}})_k \right] \right| \right. \notag  \\
&& \left. + \: \left| \sum\limits_{k \in \mc{J}_Y} \mb{Z}_{jk} \left[ \mb{i}_{\mr{I}}(\mb{v})_k - \mb{i}_{\mr{I}}(\tilde{\mb{v}})_k \right] \right| \right. \notag \\
 && \left. +\: \left| \sum\limits_{k \in \mc{J}_\Delta} \mb{Z}_{jk} \left[ \mb{i}_{\mr{I}}(\mb{v})_k - \mb{i}_{\mr{I}}(\tilde{\mb{v}})_k \right] \right| \right). \label{eqn:diff} \IEEEeqnarraynumspace 
\end{IEEEeqnarray}

The expression in \eqref{eqn:diff} is a sum of  four terms. In what follows, we  find an upper bound for each term.  

For the first term we have that
\begin{IEEEeqnarray}{rCl}
\IEEEeqnarraymulticol{3}{l}{\left| \sum\limits_{k \in \mc{J}_Y}  \mb{Z}_{jk} \left[ \mb{i}_{\mr{PQ}}(\mb{v})_k - \mb{i}_{\mr{PQ}}(\tilde{\mb{v}})_k\right] \right| } \notag\\
 &=& \left| \sum\limits_{k \in \mc{J}_Y} \mb{Z}_{jk} \left[ (\frac{s_L^k}{v_k})^* - (\frac{s_L^k}{\tilde{v}_k})^* \right] \right|  \le  \sum\limits_{k \in \mc{J}_Y}   \frac{ |\mb{Z}_{jk}|  |s_L^k|} {|v_k| |\tilde{v}_k|} |v_k - \tilde{v}_k|\notag \\
& \stackrel{\text{(a)}}{\le}&  \sum\limits_{k \in \mc{J}_Y}  \frac{ |\mb{Z}_{jk}| |s_L^k|}{ (|w_k|-R|\lambda_k|)^2} |\lambda_ku_k -\lambda_k \tilde{u}_k| \notag \\
& \stackrel{\text{(b)}}{\le} & \sum\limits_{k \in \mc{J}_Y} \frac{ |\mb{Z}_{jk}||s_L^k|}{ |w_k|^2 (1 - R \bar{\lambda}/ \underline{w})^2} |\lambda_k| \|\mb{u} - \tilde{\mb{u}}\|_{\infty}, \IEEEeqnarraynumspace \label{eqn:iypqdiff}
\end{IEEEeqnarray}
where inequality (a) is due to \eqref{eqn:linetoneutralvoltagesLowerBound1} and (b) is due to \eqref{eqn:linetoneutralvoltagesLowestBound}. 

Using Lemma \ref{lemma1}  for the second term in \eqref{eqn:diff} yields:
\begin{IEEEeqnarray}{rCl}
&&\left| \sum\limits_{k \in \mc{J}_\Delta}  \mb{Z}_{jk} \left[ \mb{i}_{\mr{PQ}}(\mb{v})_k - \mb{i}_{\mr{PQ}}(\tilde{\mb{v}})_k\right] \right| \notag \\
 &=& \left| \sum\limits_{k \in \mc{J}_\Delta} \mb{Z}_{jl_k}^{\Delta} \left[ \left(\frac{s_L^k}{\mb{e}_k^T\mb{v}_n}\right)^* - \left(\frac{s_L^k}{\mb{e}_k^T \tilde{\mb{v}}_n}\right)^* \right] \right| \notag \\
& \le & \sum\limits_{k \in \mc{J}_\Delta}   \frac{ |\mb{Z}_{jl_k}^{\Delta}|  |s_L^k|} {|(\mb{e}_k^T \mb{v}_n)| |(\mb{e}_k^T \tilde{\mb{v}}_n)|} |\mb{e}_k^T (\mb{v}_n - \tilde{\mb{v}}_n)| \notag \\\
& \stackrel{\text{(a)}}{\le}&  \sum\limits_{k \in \mc{J}_\Delta}  \frac{ |\mb{Z}_{jl_k}^{\Delta}| |s_L^k| \|\mb{e}_k\|_1 \|\mb{v}_n -\tilde{\mb{v}}_n\|_{\infty}}{ (|\mb{e}_k^T \mb{w}_n| - 2R \max_{l \in \mc{J}_n} \{ |\lambda_l| \})^2}  \notag \\
&\le & \sum\limits_{k \in \mc{J}_\Delta} \frac{ 2|\mb{Z}_{jl_k}^{\Delta}| |s_L^k| \| \mb{v}_n - \tilde{\mb{v}}_n \|_{\infty} }{|\mb{e}_k^T\mb{w}_n|^2 (1 - 2R \max_{l \in \mc{J}_n} \{ |\lambda_l| \} / |\mb{e}_k^T\mb{w}_n|)^2} \notag \\
& \stackrel{\text{(b)}}{\le} & \sum\limits_{k \in \mc{J}_\Delta} \frac{ 2|\mb{Z}_{jl_k}^{\Delta}| |s_L^k|}{|\mb{e}_k^T\mb{w}_n|^2 (1 - 2R \bar{\lambda} / \underline{\rho})^2} \max_{l \in \mc{J}_n} \{ |\lambda_l| \} \| \mb{u}_n - \tilde{\mb{u}}_n \|_{\infty}\notag \\
& \le & \sum\limits_{k \in \mc{J}_\Delta} \frac{ 2|\mb{Z}_{jl_k}^{\Delta}| |s_L^k|}{|\mb{e}_k^T\mb{w}_n|^2 (1 - 2R \bar{\lambda} / \underline{\rho})^2} \max_{l \in \mc{J}_n} \{ |\lambda_l| \} \| \mb{u} - \tilde{\mb{u}} \|_{\infty}, \IEEEeqnarraynumspace \label{eqn:ideltapqdiff}
\end{IEEEeqnarray}
where $n=\texttt{Node}[k]$,  (a) is due to H\"older's inequality and \eqref{eqn:linetolinevoltagesLowerBound1}, and (b) is  due to  \eqref{eqn:linetolinevoltagesLowestBound}.

For the third term in \eqref{eqn:diff} we have that:
\begin{IEEEeqnarray}{rCl}
&& \left| \sum\limits_{k \in \mc{J}_Y} \mb{Z}_{jk} \left[ \mb{i}_\mr{I} (\mb{v})_k - \mb{i}_\mr{I} (\tilde{\mb{v}})_k \right]\right|  \notag \\
&=& \left| \sum\limits_{k \in \mc{J}_Y} \mb{Z}_{jk} \left[ \frac{v_k}{|v_k|} i_L^k - \frac{\tilde{v}_k}{|\tilde{v}_k|} i_L^k \right] \right|  \notag \\
&\stackrel{\text{(a)}}{\le} &  \sum\limits_{k \in \mc{J}_Y} \frac{ 2|\mb{Z}_{jk}| |i_L^k|}{|v_k|} |v_k - \tilde{v}_k| \stackrel{\text{(b)}}{\le} \sum\limits_{k \in \mc{J}_Y} \frac{2|\mb{Z}_{jk}| |i_L^k|}{|w_k|- R|\lambda_k|} |v_k - \tilde{v}_k| \notag \\
& \stackrel{\text{(c)}}{\le} & \sum\limits_{k \in \mc{J}_Y} \frac{2|\mb{Z}_{jk}| |i_L^k|}{|w_k|( 1- R\bar{\lambda} / \underline{w})} |\lambda_k| \|\mb{u} - \tilde{\mb{u}}\|_{\infty}, \label{eqn:iyIdiff} 
\end{IEEEeqnarray}
where inequality (a) is due to Lemma \ref{lemma3}, (b) is due to \eqref{eqn:linetoneutralvoltagesLowerBound1}, and (c) is due to \eqref{eqn:linetoneutralvoltagesLowestBound} and the fact that $\|u_k - \tilde{u}_k\| \le \|\mb{u}- \tilde{\mb{u}}\|_{\infty}$. 

For the fourth term, Lemma \ref{lemma2} is used to obtain:
\begin{IEEEeqnarray}{rCl}
&& \left| \sum\limits_{k \in \mc{J}_\Delta} \mb{Z}_{jk} \left[ \mb{i}_\mr{I} (\mb{v})_k - \mb{i}_\mr{I} (\tilde{\mb{v}})_k \right]\right|  \notag \\
&=& \left| \sum\limits_{k \in \mc{J}_\Delta} \mb{Z}_{jl_k}^{\Delta} \left[ \frac{\mb{e}_k^T \mb{v}_n}{|\mb{e}_k^T\mb{v}_n|} i_L^k - \frac{\mb{e}_k^T\tilde{\mb{v}}_n}{|\mb{e}_k^T\tilde{\mb{v}}_n|} i_L^k \right] \right|  \notag \\
&\stackrel{\text{(a)}}\le &  \sum\limits_{k \in \mc{J}_\Delta} \frac{2 |\mb{Z}_{jl_k}^{\Delta}| |i_L^k|}{|\mb{e}_k^T\mb{v}_n|} |\mb{e}_k^T\mb{v}_n - \mb{e}_k^T\tilde{\mb{v}}_n| \notag \\
& \stackrel{\text{(b)}}{\le}& \sum\limits_{k \in \mc{J}_\Delta} \frac{2|\mb{Z}_{jl_k}^{\Delta}| |i_L^k| \|\mb{e}_k\|_1 \|\mb{v}_n - \tilde{\mb{v}}_n\|_{\infty}}{|\mb{e}_k^T \mb{w}_n|-2 R\max_{l \in \mc{J}_n} \{ |\lambda_l| \}}   \notag \\
& \stackrel{\text{(c)}}{\le}& \sum\limits_{k \in \mc{J}_\Delta} \frac{4|\mb{Z}_{jl_k}^{\Delta}| |i_L^k|}{|\mb{e}_k^T \mb{w}_n|(1-2 R \bar{\lambda} / \underline{\rho})}    \max_{l \in \mc{J}_n}\{ |\lambda_l|\} \|\mb{u}_n - \tilde{\mb{u}}_n\|_{\infty} \notag \\
& \le& \sum\limits_{k \in \mc{J}_\Delta} \frac{4|\mb{Z}_{jl_k}^{\Delta}| |i_L^k|}{|\mb{e}_k^T \mb{w}_n|(1-2 R \bar{\lambda} / \underline{\rho})}    \max_{l \in \mc{J}_n}\{ |\lambda_l|\} \|\mb{u} - \tilde{\mb{u}}\|_{\infty} , \label{eqn:ideltaIdiff}
\end{IEEEeqnarray}
where $n=\texttt{Node}[k]$, (a) is due to Lemma \ref{lemma3},   (b) is due to \eqref{eqn:linetolinevoltagesLowerBound1} and H\"older's inequality, and (c) is due to \eqref{eqn:linetolinevoltagesLowestBound}. 

Using \eqref{eqn:iypqdiff}, \eqref{eqn:ideltapqdiff}, \eqref{eqn:iyIdiff}, and \eqref{eqn:ideltaIdiff} in \eqref{eqn:diff} yields: 
\begin{IEEEeqnarray}{rCl}
\IEEEeqnarraymulticol{3}{l}{|\mb{T}(\mb{u})_j - \mb{T}(\tilde{\mb{u}})_j| \le \frac{1}{|\lambda_j|} \left[ \sum\limits_{k \in \mc{J}_Y} \frac{ |\mb{Z}_{jk}||s_L^k| |\lambda_k|}{ |w_k|^2 (1 - R \bar{\lambda}/ \underline{w})^2}  \right.} \notag \\
&&\left.  + \: \sum\limits_{k \in \mc{J}_\Delta} \frac{ 2|\mb{Z}_{jl_k}^{\Delta}| |s_L^k|\max_{l \in \mc{J}_{\texttt{Node}[k]}} \{ |\lambda_l| \} }{|\mb{e}_k^T\mb{w}_{\texttt{Node}[k]}|^2 (1 - 2R \bar{\lambda} / \underline{\rho})^2} \right. \notag \\
&&\left. + \: \sum\limits_{k \in \mc{J}_Y} \frac{2|\mb{Z}_{jk}| |i_L^k| |\lambda_k| }{|w_k|( 1- R\bar{\lambda} / \underline{w})} \right. \notag \\
&&\left. + \: \sum\limits_{k \in \mc{J}_\Delta} \frac{4|\mb{Z}_{jl_k}^{\Delta}| |i_L^k| \max_{l \in \mc{J}_{\texttt{Node}[k]}}\{ |\lambda_l|\}}{|\mb{e}_k^T \mb{w}_{\texttt{Node}[k]}|(1-2 R \bar{\lambda} / \underline{\rho})}   \right] \| \mb{u}- \tilde{\mb{u}}\|_{\infty} \IEEEeqnarraynumspace \label{eqn:diffdistance1}
\end{IEEEeqnarray}
For  $\mb{T}(\mb{u})$ to be a contraction, the coefficient multiplying $\|\mb{u} - \tilde{\mb{u}}\|_{\infty}$ must be less than $1$ for all $j \in \mc{J}$, which is ensured by \eqref{eqn:contractioncondition}. 
%
%

\ifCLASSOPTIONcaptionsoff
  \newpage
\fi
\bibliographystyle{IEEEtran}
\bibliography{references}
\vspace{\baselineskip}

\begin{IEEEbiographynophoto}{Mohammadhafez Bazrafshan}
received the B.S. degree  from Iran University of Science and Technology in 2012, and the M.S. degree in 2014 from  the University of Texas at San Antonio  both in electrical engineering, where he is  currently a PhD student with research interests in smart power grids. 
\end{IEEEbiographynophoto}
\vspace{-17cm}
\begin{IEEEbiographynophoto}{Nikolaos Gatsis} received the Diploma degree in Electrical and Computer Engineering from the University of Patras, Greece, in 2005 with honors. He completed his graduate studies at the University of Minnesota, where he received the M.Sc. degree in Electrical Engineering in 2010, and the Ph.D. degree in Electrical Engineering with minor in Mathematics in 2012. He is currently an Assistant Professor with the Department of Electrical and Computer Engineering at the University of Texas at San Antonio. His research interests lie in the areas of smart power grids, renewable energy management, communication networks, and cyberphysical systems, with an emphasis on optimal resource management. Prof. Gatsis co-organized symposia in the area of Smart Grids in IEEE GlobalSIP 2015 and IEEE GlobalSIP 2016. He also served as a Technical Program Committee member for symposia in IEEE SmartGridComm 2013, IEEE SmartGridComm 2014, and IEEE SmartGridComm 2015.
\end{IEEEbiographynophoto}

\end{document}